\documentclass[12pt,reqno,tbtags]{amsart}
\usepackage{comment}

\usepackage{amsmath}
\usepackage{mathtools}
\usepackage{amssymb}
\usepackage{amscd}
\usepackage{amsfonts}
\usepackage{setspace}
\usepackage{version}
\usepackage{enumitem}
\usepackage{xcolor}
\usepackage{url}

\usepackage{tikz-cd}
\usepackage{tikz}

\usepackage[pdfencoding=auto]{hyperref}
\hypersetup{
  colorlinks,
  linkcolor={blue!50!black},
  citecolor={green!40!black},
  urlcolor={blue!80!black}
}

\usepackage[capitalize]{cleveref}

\newtheorem{theorem}{Theorem}[section]
\newtheorem{lemma}[theorem]{Lemma}
\newtheorem{proposition}[theorem]{Proposition}
\newtheorem{corollary}[theorem]{Corollary}
\newtheorem{conjecture}[theorem]{Conjecture}

\theoremstyle{definition}

\theoremstyle{remark}
\newtheorem{remark}[theorem]{Remark}

\renewcommand{\leq}{\leqslant}
\renewcommand{\le}{\leq}
\renewcommand{\geq}{\geqslant}

\newcommand\D{\operatorname{D}}
\def\F{\mathbf{F}}
\def\R{\mathbf{R}}
\def\C{\mathbf{C}}
\def\Z{\mathbf{Z}}

\def\P{\mathbf{P}}

\newcommand{\ent}[1]{\mathbf{H}[#1]}
\newcommand{\I}[1]{\mathbf{I}[#1]}
\newcommand{\bigI}[1]{\mathbf{I}\big[#1\big]}
\newcommand{\bigH}[1]{\mathbf{H}\big[#1\big]}

\def\eps{\varepsilon}
\newcommand{\dist}[2]{{\operatorname{d}[#1;#2]}} % This version doesn't have the spacing beforehand
\newcommand{\costa}[1]{\eta\bigl(\dist{X_1^0}{#1} - \dist{X_1^0}{X_1}\bigr)} %\chktex 8
\newcommand{\costb}[1]{\eta\bigl(\dist{X_2^0}{#1} - \dist{X_2^0}{X_2}\bigr)} %\chktex 8
\newcommand{\bigdist}[2]{{\operatorname{d}\bigl[#1;#2\bigr]}} % This version doesn't have the spacing beforehand

\numberwithin{equation}{section}

\usepackage{todonotes}

\begin{document}

\title[On a conjecture of Marton]{On a conjecture of Marton}


%    author one information

\author{W.~T.~Gowers}
\thanks{}
\address{Coll\`ege de France \\
11, place Marcelin-Berthelot \\
75231 Paris 05 \\
France \\
and Centre for Mathematical Sciences \\
Wilberforce Road \\
Cambridge CB3 0WB \\
UK}
\email{wtg10@dpmms.cam.ac.uk}

\author{Ben Green}
\thanks{BG is supported by Simons Investigator Award 376201.}
\address{Mathematical Institute \\
Andrew Wiles Building \\
Radcliffe Observatory Quarter \\
Woodstock Rd \\
Oxford OX2 6QW \\
UK}
\email{ben.green@maths.ox.ac.uk}

\author{Freddie Manners}
\thanks{FM is supported by a Sloan Fellowship.  During part of the preparation of this work, he is grateful for the hospitality of the Simons Institute for the Theory of Computing.}
\address{Department of Mathematics \\
University of California, San Diego (UCSD)\\
9500 Gilman Drive \# 0112 \\
La Jolla, CA  92093-0112 \\ %chktex 8
USA}
\email{fmanners@ucsd.edu}

\author{Terence Tao}
\thanks{TT is supported by NSF grant DMS-1764034.  }
\address{Math Sciences Building \\
520 Portola Plaza \\
Box 951555 \\
Los Angeles, CA 90095 \\
USA}

\email{tao@math.ucla.edu}

\subjclass[2000]{Primary }

\begin{abstract}
We prove a conjecture of K.~Marton, widely known as the polynomial Freiman--Ruzsa conjecture, in characteristic $2$. The argument extends to odd characteristic, with details to follow in a subsequent paper.
\end{abstract}

\maketitle

\section{Introduction}
In this paper we prove a conjecture of Katalin Marton (see~\cite{ruzsa-fin-field}), widely known in the literature as the \emph{polynomial Freiman--Ruzsa conjecture} in characteristic $2$.
The conjecture has many equivalent forms, one of which is the following statement.

\begin{conjecture}\label{pfr-conj}
 Suppose that $A \subset \F_2^n$ is a set with $|A + A| \leq K|A|$. Then $A$ is covered by at most $2 K^{C}$ cosets\footnote{The factor of $2$ is needed to deal with the case where $K$ is close to $1$.  If $H_0$ is a subgroup of $\F_2^n$ and $A \subseteq H_0$ is a subset of size $(1-\eps)|H_0|$ for some small $\eps>0$, then in order to obtain the bound $|H|\leq|A|$, we need to pick a subgroup $H$ of $H_0$ of index 2 and cover $A$ by the two cosets of that subgroup.} of some subgroup $H \leq \F_2^n$ of size at most $|A|$.
\end{conjecture}

\noindent Note that if $A$ is covered by at most $r$ cosets of a subgroup $H$ of size at most $|A|$, then $|A+A|\leq (\binom r2+1)|A|$, so \Cref{pfr-conj} allows one to pass from the property of having a small doubling constant to the property of being contained in a small union of cosets of a subgroup of size at most $|A|$ and back again with at most a polynomial loss in the parameters.

The first result in this direction was due to Imre Ruzsa~\cite{ruzsa-fin-field}, who obtained an upper bound of $2K^2 2^{K^4}$ in place of the desired $2K^C$. It was Ruzsa~\cite{ruzsa-fin-field},~\cite[Conjecture 2.2.2]{ruzsa-sumsets-survey} who attributed \Cref{pfr-conj} to Marton. Sanders~\cite{sanders} made a dramatic breakthrough by proving the first bounds of shape $\ll \exp(\log^{C_1} K)$, showing that $C_1 = 4 + \eps$ is permissible here. A variant of Sanders' argument due to Konyagin (not published by Konyagin, but described by Sanders in~\cite{sanders2}, and see also the comments on page 8 of~\cite{gmt}) showed that $C_1 = 3 + \eps$ is permissible. We also remark that \Cref{pfr-conj} (with worse values of $C$) was previously established in the special case when $A$ was a downset in~\cite{gt-freiman}. We refer to~\cite{green-fin-fields, green-pfr-note, lovett-survey} for surveys on this conjecture.

Our main result is the following.

\begin{theorem}\label{pfr}
\Cref{pfr-conj} is true with $C = 12$.
\end{theorem}

An elaboration of our arguments gives corresponding results in vector spaces over $\F_p$, $p$ odd. This is notationally somewhat heavier and makes the underlying ideas slightly harder to appreciate so we give this argument in a separate paper~\cite{ggmt-pfr-odd}.

Since the initial release of this preprint, Jyun-Jie Liao has given a refinement of the argument that improves the constant $C=12$ in \Cref{pfr} to $C=11$ (with corresponding numerical improvements to our other main results and applications); see~\cite{liao1,liao2}.

\subsubsection*{Applications}
It was shown in~\cite[Theorem 1.11]{gmt} that \Cref{pfr-entropy} (or \Cref{pfr}) in characteristic $2$ implies the so-called `weak' polynomial Freiman--Ruzsa conjecture over $\Z$, so that is now a theorem as well.

\begin{theorem}\label{th13}  Let $A$ be a finite subset of $\Z^D$ for some $D$, and suppose that $|A+A| \leq K|A|$. Then there is some set $A' \subseteq A$, $|A'| \geq K^{-C_1/2} |A|$, with $\dim A' \leq C_2 \log K$, for some absolute\footnote{If one traces through the arguments in~\cite{gmt} (dropping all terms involving the exponent $1-1/C_{\mathrm{PFR}}$, and using the constant $C=11$ from \Cref{pfr-entropy} below), one finds that one can take $C_1 = 20C = 220$ and $C_2 = \frac{40}{\log 2}$. The reason we have stated \Cref{th13} with a $C_1/2$ in the exponent is so that the constants $C_1, C_2$ here are the same as the ones in~\cite{gmt}.} constants $C_1,C_2>0$.
\end{theorem}

\Cref{th13} solves a conjecture stated in~\cite{chang, gmt, manners,pz}. We remark that there is also a `strong' PFR conjecture over $\Z$ (though its formulation requires care; see~\cite{lovett-regev} and~\cite{manners-pfr-formulation} for more discussion). This remains a challenging open problem.

As mentioned above, the polynomial Freiman--Ruzsa conjecture is equivalent to many other statements in additive combinatorics. We mention a few of these here.

\begin{corollary}\label{additive-stability}
There is a constant $C_3$ such that if $f: \F_2^m \to \F_2^n$ is a function such that the set $\{ f(x)+f(y)-f(x+y): x,y \in \F_2^m\}$ has cardinality at most $K$, then $f$ may be written as $g+h$, where $g$ is linear and the range of $h$ is a set of cardinality $O(K^{C_3})$. \end{corollary}

\noindent The equivalence of this and \Cref{pfr} is due to Ruzsa~\cite[Conjecture 2.2.3]{ruzsa-sumsets-survey}. Ruzsa did not publish the proof, but he communicated it to the second author and the details may be found in~\cite{green-pfr-note}. Here, as usual, $O(X)$ denotes a quantity bounded by $CX$ for some absolute constant~$C$. 

\begin{corollary}\label{partial-additivity}
There is a constant $C_4$ with the following property. Suppose that $f : \F_2^m \to \F_2^n$ is a function such that, if $x, y$ are selected uniformly at random from $\F_2^m$, $\P(f(x + y) = f(x) + f(y)) \geq 1/K$. Then there is a homomorphism $g \colon \F_2^m \to \F_2^n$ such that $\P(f(x) = g(x)) \gg K^{-C_4}$.
\end{corollary}

\noindent See~\cite[Theorem 4.1]{samorodnitsky}. The notation $X\gg Y$ means that there is an absolute constant $c>0$ such that $X\geq cY$.

The next result is a polynomially effective inverse theorem for the Gowers uniformity norm ${U^3(\F_2^n)}$; see~\cite{gt-equiv, lovett, samorodnitsky}.

\begin{corollary}\label{u3-inverse}
There is a constant $C_5$ with the following property. Let $f \colon \F_2^n \to \C$ be a $1$-bounded function such that $\|f\|_{U^3(\F_2^n)} \geq 1/K$. Then there exists a quadratic polynomial $P \colon \F_2^n \to \F_2$ such that \[ |\frac{1}{2^n} \sum_{x \in \F_2^n} f(x) (-1)^{P(x)}| \gg K^{-C_5}.\]
\end{corollary}

\noindent In principle, explicit values for the exponents $C_3,C_4, C_5$ (and for the unspecified implied constants in the big-O and $\gg$ notations) could be worked out by tracing through the arguments in the references provided, but we will not do so here.

For our arguments a particularly relevant equivalent formulation of the PFR conjecture is the entropic
version of the conjecture from~\cite{gmt}, which we present with explicit constants (\Cref{pfr-entropy} below).

Turning to applications of a more miscellaneous nature, a particularly attractive application of \Cref{th13} is the following result of Mudgal, which is a special case of~\cite[Theorem 6.2]{mudgal}.

\begin{corollary}[Mudgal]
Let $A \subset \R$ be a finite set with at least two elements. Then either then $m$-fold sumset $mA$ or the $m$-fold product set $A^{(m)}$ has cardinality at least $|A|^{f(m)}$, where $f(m) \rightarrow \infty$ as $m \rightarrow \infty$.
\end{corollary}
This generalizes a well-known result of Bourgain and Chang~\cite{bourgain-chang}, who established this when $A$ is a set of integers. More general results of a similar flavour may be found in~\cite{mudgal}.

The PFR conjecture has several applications to topics in theoretical computer science such as linearity testing (see \Cref{partial-additivity}), construction of extractors, communication complexity, locally decodable codes, and non-malleable codes; see~\cite{lovett-survey} for a survey.

\subsubsection*{Entropy methods}
Ruzsa~\cite{ruzsa-entropy} and the fourth author~\cite{tao-entropy} studied entropy analogues of the notion of sumset. In a recent preprint~\cite{gmt}, the second, third and fourth authors further developed the use of entropy techniques in additive combinatorics. Entropy methods, as well as several of the results in both~\cite{gmt} and~\cite{tao-entropy}, will be used extensively here.

For a full discussion of the relevant concepts, we refer the reader to~\cite{gmt} or to \Cref{entropy-app}. For now, we recall the notion of entropic Ruzsa distance between random variables. If $X, Y$ are two finitely supported random variables taking values in an additive group $G$, then we define
\begin{equation}
\label{ruz-dist-def} \dist{X}{Y} \coloneqq \ent{X' - Y'} - \tfrac{1}{2} \ent{X'} - \tfrac{1}{2} \ent{Y'},
\end{equation}
where $X', Y'$ are independent copies of $X, Y$.
In particular we stress that $\dist{X}{Y}$ depends only on the \emph{individual distributions} of $X$ and $Y$: it does not require them to be independent, or even that they are defined on a common sample space.

We refer the reader to~\cite[Section 1]{gmt} and \Cref{entropy-app} for further discussion. For the moment we just mention the \emph{Ruzsa triangle inequality} $\dist{X}{Y} \leq \dist{X}{Z} + \dist{Z}{Y}$, as well as the fact that, despite $\dist{-}{-}$ being called a distance, we have $\dist{X}{X} = 0$ only when $X$ is a coset of a subgroup (see \Cref{lem:100pc} below). In this paper (since we will use no other notion of distance) we write $\dist{X}{Y}$ instead of the more cumbersome $d_{\mathrm{ent}}(X,Y)$ from~\cite{gmt}.
We also switch to using square brackets, as many of our random variables $X, Y$ are themselves expressions involving parentheses.

One consequence of the investigation in~\cite{gmt} was another equivalent formulation of the PFR conjecture. This is the version we will focus on in the present paper.

\begin{theorem}\label{pfr-entropy}
  Let $G = \F_2^n$, and suppose that $X^0_1, X^0_2$ are $G$-valued random variables.
  Then there is some subgroup $H \leq G$ such that
  \[
    \dist{X^0_1}{U_H} + \dist{X^0_2}{U_H} \le 11\dist{X^0_1}{X^0_2},
  \]
  where $U_H$ denotes the uniform distribution on $H$.
  Furthermore, both $\dist{X^0_1}{U_H}$ and $\dist{X^0_2}{U_H}$ are at most $6 \dist{X^0_1}{X^0_2}$.
\end{theorem}

The fact that \Cref{pfr-entropy} is equivalent, up to constants, to the characteristic $2$ case of Marton's original conjecture, as well as to the other combinatorial formulations of the polynomial Freiman--Ruzsa conjecture, was established in~\cite[Section 8]{gmt}.  For the convenience of the reader, we give the proof that \Cref{pfr-entropy} implies \Cref{pfr} in \Cref{equiv-app}.

\subsubsection*{Variants and further remarks}
The polynomial Freiman--Ruzsa conjecture has inspired various analogous investigations in other contexts; for a recent example in the symmetric group, see~\cite{keevash-lifshitz}.

One could consider strengthening \Cref{additive-stability} by requiring the function $h$ to take values in a $C_7$-fold sumset of the set $\{ f(x)+f(y)-f(x+y): x,y \in \F_2^m\}$, for some absolute constant $C_7$.
However, this ``strong PFR conjecture'' is known to be false: see~\cite{aaronson} and~\cite[Section 1.17]{tao-poincare}.

The \emph{polynomial Bogolyubov conjecture} (see e.g.,~\cite[Conjecture 1.5]{lovett-survey}) asserts that if $A$ is a non-empty subset of $\F_2^n$ of doubling at most $K$, then $2A-2A=4A$ contains a subspace of cardinality $\gg K^{-O(1)} |A|$.  In~\cite{sanders} it was shown that $2A-2A$ contains a subspace of size at least $\gg K^{-O(\log^3 K)} |A|$; see also a recent bilinear extension of this result in~\cite{hosseini-lovett}.  We were unable to use our methods to improve upon these results.  However, we can obtain a weak Bogulyubov type result, in which the number of iterates in the sumset is allowed to grow logarithmically in $K$; see \Cref{bog-remark}.

Let $f(2,K)$ denote the best constant such that every non-empty subset $A$ of $\F_2^n$ of doubling at most $K$ is contained in an affine subspace of cardinality at most $f(2,K) |A|$.  It is known that $f(2,K)$ is a piecewise linear function that can be explicitly computed for small $K$, and has the asymptotic bounds
\[ \frac{2^{2K}}{4K} (1-o(1)) \leq f(2,K) \leq \frac{2^{2K}}{2K}(1+o(1))\]
as $K \to \infty$; see~\cite{even-zohar} (which builds upon prior work in~\cite{diao, gt-freiman, konyagin}).  The results here do not appear to significantly strengthen these bounds.

A formalization of \Cref{pfr} in the proof assistant language {\tt Lean4} may be found at \url{https://teorth.github.io/pfr/}.

\subsubsection*{Notation}
All logarithms in this paper will be natural logarithms. We will very extensively use the notation of entropy $\ent{-}$ and mutual information $\I{-}$ as well as conditional variants of these concepts. Basic definitions and facts are recalled in \Cref{entropy-app}.

\subsubsection*{Acknowledgments}

It is a pleasure to thank Floris van Doorn, Zach Hunter, Emmanuel Kowalski, and Arend Mellendijk for corrections and comments.  We are also grateful to Jyun-Jie Liao for sharing his refinement of the constant in \Cref{pfr}.  The Lean formalization project was set up by Ya\"el Dilles and the fourth author; we thank the many participants of that project (see \url{https://github.com/teorth/pfr/graphs/contributors} for a listing) for their invaluable contributions, which allowed the formalization to be completed in just three weeks (and, in particular, before this paper was submitted for publication).

\section{Induction on entropy distance}\label{sec2}

Fix, for the rest of the paper, the group $G = \F_2^n$ and the variables $X^0_1, X^0_2$ appearing in \Cref{pfr-entropy}. We shall think of $X^0_1, X^0_2$ as `reference variables', and not modify them for the rest of the proof; while they do appear in many of the expressions we will need to control, they will play a relatively minor role in the estimates.
 Our overall strategy can be thought of as a kind of induction on (a slight modification of) the distance $\dist{X_1}{X_2}$, though we will ultimately phrase it as a compactness argument.

For the rest of the paper, $\eta$ will denote the constant $\frac{1}{9}$. For any two $G$-valued random variables $X_1, X_2$, we introduce the functional
\begin{equation}\label{tau-def}
  \tau[X_1; X_2] \coloneqq \dist{X_1}{X_2} + \eta  \dist{X^0_1}{X_1} + \eta \dist{X^0_2}{X_2}.
\end{equation}
Note that this functional depends only on the distributions $p_{X_1}, p_{X_2}$ of $X_1,X_2$, and not on whether $X_1,X_2$ (or $X_1^0$, $X^0_2$) are independent or dependent (or even on whether they are defined on the same sample space).

Here is the main result from which we will deduce \Cref{pfr-entropy}.

\begin{proposition}\label{de-prop}
  Let $X_1, X_2$ be two $G$-valued random variables with $\dist{X_1}{X_2} > 0$. Then there are $G$-valued random variables $X'_1, X'_2$ such that
\begin{equation}
\label{d-opt} \tau[X'_1;X'_2] < \tau[X_1;X_2].
\end{equation}
\end{proposition}

If we ignore the two penalty terms involving $\eta$ on the right-hand side of~\eqref{tau-def}, then \Cref{de-prop} asserts that if two random variables have a positive Ruzsa distance, then we can `improve' them, finding two new variables $X'_1, X'_2$ that are closer to each other than the original variables $X_1,X_2$.  To take care of the penalty terms, we will also need the new variables $X'_1,X'_2$ to be `related' to the old variables $X_1,X_2$ in some sense, so that both sets of variables have a comparable Ruzsa distance to the reference variables $X^0_1, X^0_2$.

In order to deduce \Cref{pfr-entropy} from \Cref{de-prop} we will use the `100 percent' case of our main theorem, which we note now.
\begin{lemma}%
  \label{lem:100pc}
  Suppose that $X_1,X_2$ are $G$-valued random variables such that
  $\dist{X_1}{X_2}=0$. Then there exists a subgroup $H \leq G$ such that $\dist{X_1}{U_H} = \dist{X_2}{U_H} = 0$.
\end{lemma}
\begin{proof}%
  By the triangle inequality, $\dist{X_1}{X_1} = 0$, and so (since $G=\F_2^n$) we have $\dist{X_1}{-X_1} = 0$. By~\cite[Theorem~1.11(i)]{tao-entropy}, it follows that there exists a subgroup $H$ with $\dist{X_1}{U_H} = 0$.
  Then $\dist{X_2}{U_H}=0$ follows by the triangle inequality.
\end{proof}

\begin{proof}[Proof of \Cref{pfr-entropy}, assuming \Cref{de-prop}.]
We choose a pair of variables $X_1,X_2$ (or more precisely, their distributions on $G$) to be any minimizer\footnote{We remark that an entropy functional minimization was previously employed by the fourth author in~\cite{tao-regularity} to give an alternative proof of the Szemer\'edi regularity lemma.} of the functional $\tau$.
  This optimization problem ranges over the square of the space of probability distributions on the finite set $G$. The usual topology on this space is compact, and it is straightforward that $\dist{-}{ -}$ is continuous with respect to this topology, so the minimum exists.

As $(X_1,X_2)$ minimizes $\tau$, by the contrapositive of \Cref{de-prop} we must have $\dist{X_1}{X_2}=0$,
and hence by \Cref{lem:100pc} there exists a subgroup $H \leq G$ such that $\dist{X_1}{U_H} = \dist{X_2}{U_H}=0$.
  Finally, we have
  \begin{align*}
    \eta (\dist{X_1^0}{U_H} + & \dist{X_2^0}{U_H})= \eta (\dist{X_1^0}{X_1} + \dist{X_2^0}{X_2}) \\ & = \tau[X_1;X_2]  \leq \tau[X_2^0;X_1^0] = (1+2\eta) \dist{X_1^0}{X_2^0},
  \end{align*}
and so with our choice $\eta = \frac{1}{9}$ we obtain the first claim of \Cref{pfr-entropy}.  Since
  $
    \lvert \dist{X^0_1}{U_H} - \dist{X^0_2}{U_H} \rvert \le \dist{X^0_1}{X^0_2}
  $
  by the triangle inequality, the second statement then follows.
\end{proof}

\begin{remark}\label{quant}
  The proof of \Cref{de-prop} is completely algorithmic, but our use of compactness here means that the proof of \Cref{pfr-entropy} is not.
  However, our argument can be modified to give a completely algorithmic proof, at the expense of slightly worse constants.

  One approach is to prove \Cref{de-prop} with~\eqref{d-opt} replaced by the stronger bound
  \begin{equation}\label{proper-decrement}
    \dist{X'_1}{X'_2} \leq (1 - \eta') \dist{X_1}{X_2} - \eta' \dist{X'_1}{X_1} - \eta' \dist{X'_2}{X_2},
  \end{equation}
  for some value $\eta'$ rather smaller than $\eta=\frac{1}{9}$.
  The proof is very similar and follows the same strategy as that outlined in \Cref{sec3}.
  This variant is perhaps even slightly easier to follow than the current proof, since the `fixed' variables $X_1^0, X_2^0$ no longer play a role.

  Once~\eqref{proper-decrement} is established, it can be applied iteratively  to give a sequence $(X_1^{t}, X_2^{t})$, $t = 0,1,\dots$ for which $\dist{X_1^{t}}{X_2^{t}}$ decays exponentially.
  In particular, for some $T = O(\log \dist{X^0_1}{X^0_2})$ we have $\dist{X_1^{T}}{X_2^{T}} \leq \eps_0$, where $\eps_0$ is the absolute constant appearing in~\cite[Proposition 1.3]{gmt}.
  By that result, $X_1^{T}$ and $X_2^{T}$ are both close to some variable $U_H$. One can then control the distance of $U_H$ from the original variables $X^0_1, X^0_2$ by summing~\eqref{proper-decrement} over $0 \leq t < T$ and applying the Ruzsa triangle inequality repeatedly.

  A version of this argument will be given in detail in our subsequent paper~\cite{ggmt-pfr-odd}.
\end{remark}
\begin{remark}\label{bog-remark}
   We sketch one further consequence of the variant argument discussed in \Cref{quant} above, namely a weak Bogulyubov-type result: if $|A+A| \leq K|A|$, then there is a subspace $H$ of $G$ of cardinality $\gg K^{-O(1)} |A|$ that is contained in an iterated sumset $mA$ for some $m = O(\log^{O(1)}(2+K))$.

  Indeed, if takes $X^0_1 = X^0_2 = U_A$ as in \Cref{equiv-app} and runs the above iteration, we see inductively from an inspection of the construction that each pair $X_1^t, X_2^t$ of random variables produced by the iterative process is contained in $m_t A$ for some $m_t = O(1)^t$.  Since it only takes $O(\log\log K)$ steps for the doubling constant to get within range of $\eps_0$, one can check (by inspection of the argument in~\cite[Proposition 1.3]{gmt}) that the group $H$ produced is then contained in $mA$ for some $m = O(\log^{O(1)}(2+K))$, and the claim then follows after some routine verifications.
\end{remark}

\section{Plan of the remaining argument}%
\label{sec3}

Let $k$ denote the quantity $k \coloneqq \dist{X_1}{X_2}$.
For the following preliminary discussion we write various estimates as $O(\eta k)$ for simplicity. To make the argument actually work with the parameter $\eta = \frac{1}{9}$ one of course needs to be more precise, as indeed we will be from~\eqref{first-est} onwards.

To prove \Cref{de-prop} we will try various choices of $X'_1, X'_2$ constructed using $X_1, X_2$ and show that at least one of them works.

To describe the choices we will try, consider a four-tuple of independent random variables $X_1, X_2, \tilde X_1, \tilde X_2$ where $X_1,\tilde X_1$ are copies of $X_1$ and $X_2,\tilde X_2$ are copies of $X_2$; we can in fact assume that all six variables $X_1,X_2,\tilde X_1,\tilde X_2,X^0_1,X^0_2$ are independent. Our primary choices for $(X'_1, X'_2)$ will be sums
\begin{equation}
  \label{sum-1}
  X'_1 = X_1 + \tilde X_2,\qquad X'_2 = X_2 + \tilde X_1
\end{equation}
or
\begin{equation}
  \label{sum-2}
  X'_1 = X_1 + \tilde X_1,\qquad X'_2 = X_2 + \tilde X_2
\end{equation}
or alternatively  `fibres'
\begin{equation}
  \label{fiber-1}
  X'_1 = (X_1 | X_1 + \tilde X_2= g),\qquad X'_2 = (X_2 | X_2 + \tilde X_1 = g')
\end{equation}
or
\begin{equation}
  \label{fiber-2}
  X'_1 = (X_1 | X_1 + \tilde X_1= g),\qquad X'_2 = (X_2 | X_2 + \tilde X_2 = g')
\end{equation}
for $g, g' \in G$.
A result that we call the \emph{fibring lemma} (which is~\cite[Proposition 1.4]{gmt}), concerning the behaviour of entropy doubling under homomorphisms, may be used to relate the distances that come up when we choose the $X'_i$ to be sums or fibres.
A self-contained account of this is provided in \Cref{sec4}.

The conclusion is that either one of these choices succeeds, or they all \emph{narrowly} (by $O(\eta k)$) fail to work, and the latter can happen only if all the inequalities we used were close to equalities. Motivated by this, one can go back and analyse what happens if the inequality in the fibring lemma was almost tight.
What we find in this case is that if neither choice~\eqref{sum-1},~\eqref{fiber-1} works to prove~\eqref{d-opt} then we have an upper bound
\begin{equation}\label{first-est-a}
  I_1 \coloneqq \bigI{ X_1+X_2 : \tilde X_1 + X_2 | X_1+X_2+\tilde X_1+\tilde X_2 } = O(\eta k).
\end{equation}
Informally, this says that $X_1 + X_2$ and $\tilde X_1 + X_2$ are almost independent conditioned on $X_1+X_2+\tilde X_1+\tilde X_2$.
Similarly, if neither choice~\eqref{sum-2},~\eqref{fiber-2} works to prove~\eqref{d-opt}, we obtain a bound
\begin{equation}\label{second-est-a}
  I_2 \coloneqq \bigI{ X_1+X_2 : X_1 + \tilde X_1 | X_1+X_2+\tilde X_1+\tilde X_2 } = O(\eta k).
\end{equation}
Finally,
\begin{equation}\label{third-est-a}
 I_3 \coloneqq \bigI{ \tilde X_1+X_2 : X_1 + \tilde X_1 | X_1+X_2+\tilde X_1+\tilde X_2 } = O(\eta k),
 \end{equation}
 since in fact we have $I_2 = I_3$, as may be seen by interchanging the names of $\tilde X_1$ and $X_1$.

If none of the choices~\eqref{sum-1},~\eqref{sum-2},~\eqref{fiber-1},~\eqref{fiber-2} works to prove~\eqref{d-opt}, all three estimates~\eqref{first-est-a},~\eqref{second-est-a} and~\eqref{third-est-a} hold.
In this case, we proceed to a part of the argument we refer to as the `endgame'.

Suppose first for simplicity that the mutual informations in~\eqref{first-est-a},\ \eqref{second-est-a} and~\eqref{third-est-a} were zero rather than merely small.
It follows that for any $s$ in the support of $X_1+X_2+ \tilde X_1+\tilde X_2$, the three random variables
\begin{align}
  \nonumber
  T_1&= (X_1+X_2 |X_1+X_2+\tilde X_1+\tilde X_2=s),\\
  \nonumber
  T_2&=(X_1+\tilde X_1 |X_1+X_2+\tilde X_1+\tilde X_2=s), \\
  T_3&=(\tilde X_1+X_2 |X_1+X_2+\tilde X_1+\tilde X_2=s)
  \label{three-variables}
\end{align}
are pairwise independent.
We also note that $T_1+T_2+T_3$ is constant (in fact it is identically zero).
However, for any trio $(T_1,T_2,T_3)$ of pairwise independent random variables such that $T_1+T_2+T_3$ is constant, a short calculation shows that
\[\dist{T_1}{T_2}  = \ent{T_3} - \tfrac{1}{2} \ent{T_1} - \tfrac{1}{2} \ent{T_2} \]
and similarly for cyclic permutations, hence
\begin{equation}
  \label{eq:i-eq-0}
  \dist{T_1}{T_2} + \dist{T_1}{T_3} + \dist{T_2}{T_3} = 0
\end{equation}
and so, taking $X'_1$ and $X'_2$ to be some pair of $T_1,T_2,T_3$, we have $\dist{X_1'}{X_2'} = 0$.
This is a very strong conclusion, and with some further estimation, it allows us to establish~\eqref{d-opt}. (The choice of $X'_1,X'_2$ here is inspired, albeit rather indirectly, by a step in an argument of Katz and Koester~\cite{kk}.)

If the three variables in~\eqref{three-variables} are merely `almost pairwise independent', as in~\eqref{first-est-a} and~\eqref{second-est-a}, then things are less straightforward.
Very roughly, sums of non-independent random variables are to sums of independent random variables as \emph{partial} sumsets $A+_E B$ along a bipartite graph $E$ are to full sumsets $A+B$.
Hence, it is natural to apply a variant of the entropic Balog--Szemer\'edi--Gowers lemma due to the fourth author~\cite[Lemma 3.3]{tao-entropy} (reproduced as \Cref{lem-bsg} here) to pass from the `almost independent' variables $(T_1,T_2)$ as above, for which $\ent{T_1+T_2}-\tfrac12 \ent{T_1} - \tfrac12 \ent{T_2}$ is small (or even negative), to variables $X_1',X_2'$ with $\dist{X_1'}{X_2'} = O(\eta k)$.
These variables $X'_1,X'_2$, which are given to us by \Cref{lem-bsg}, now establish~\eqref{d-opt}.

Examining the proof of \Cref{lem-bsg} carefully, the final choice of $X'_1$ that comes out of this argument is of the form
\begin{equation}
  \label{from-bsg}
  (X_1 + X_2 \,|\, \tilde X_1+ X_2 = t, X_1 + \tilde X_2 = s+t)
\end{equation}
or a related quantity obtained by permuting the variables, and similarly for $X'_2$ (using the same $s,t$ for both $X_1'$ and $X_2'$).

\subsubsection*{Motivating examples}
To understand the strategy, it is helpful to consider some cases of the form $X_1 = U_{A_1}$, $X_2 = U_{A_2}$ for various sets $A_1,A_2 \subseteq \F_2^n$, and to discuss which choices of $X'_1, X'_2$ give the desired estimate~\eqref{d-opt}. In this discussion is convenient to write $K \coloneqq e^k$.

Consider first the case in which $A_1=A_2=A$ and $A$ is a random subset of some subgroup $H \le \F_2^n$ with density $K^{-1}$ in $H$.
Then (almost surely) $\dist{X_1}{X_2} \approx k$.
In this case the choice~\eqref{sum-2}, that is to say $X'_1 = X_1 + \tilde X_1$ and $X'_2 = X_2 + \tilde X_2$, immediately establishes~\eqref{d-opt}. Indeed both $X'_1$ and $X'_2$ are close to the uniform distribution $U_H$ on $H$, so $\dist{X'_1}{X'_2} \approx 0$ and $\dist{X'_i}{X_i}\approx k/2$ for $i=1,2$, and~\eqref{d-opt} follows, with room to spare, by the triangle inequality.

Consider next the case in which $A_1 = \bigcup_{i = 1}^{m} (x_i + H)$ and $A_2 = \bigcup_{i=1}^m (y_i+ H)$, where $H < \F_2^n$ is a subgroup and the $x_i+H$ and $y_i+H$ are linearly independent in $G/H$. Setting $m \coloneqq K$ gives $\dist{X_1}{X_2} \approx k$.
In this case one can obtain~\eqref{d-opt} with the choice~\eqref{fiber-1}, that is to say $X'_1 = (X_1 | X_1 + \tilde X_2 = g)$ and $X'_2 = (X_2 | X_2 + \tilde X_1 = g')$ for any $g,g' \in A_1+A_2$.
Indeed, $X'_1$ is the uniform distribution on $A_1 \cap (A_2 + g)$, which if $g \in x_i+y_j+H$ is exactly the coset $x_i+H$ of $H$ (by linear independence). Similarly, $X'_2$ is uniform on a coset $y_j'+H$, so $\dist{X'_1}{X'_2} \approx 0$ and again~\eqref{d-opt} follows.
(Using the fibres in~\eqref{fiber-2} also gives a contradiction in this case, but the analysis is slightly more involved, since $A_1 \cap (A_1 + g)$ is typically a union of two cosets of $H$.)

In both of the above examples, one of the choices~\eqref{sum-1},~\eqref{sum-2},~\eqref{fiber-1} or~\eqref{fiber-2} already gives the estimate~\eqref{d-opt} and so it is not necessary to proceed to the `endgame'.

Here is a third example, a sort of combination of the previous two, for which the choices~\eqref{sum-1},~\eqref{sum-2},~\eqref{fiber-1} and~\eqref{fiber-2} all fail to work.
Let $A_1',A_2'$ be (independent) random subsets of density $1/m$ of the sets $A_1,A_2$ from the previous example, where now $m = \sqrt{K}$.
One may then check that $X_1 + \tilde X_1, X_2 + \tilde X_2$ resemble the uniform distribution on the union of $\approx K/2$ cosets of $H$, so $\dist{X'_1}{X'_2} \approx k - O(1)$ and we do not obtain~\eqref{d-opt} (if $k$ is large) with the choice~\eqref{sum-2}. The case of~\eqref{sum-1} is similar.

On the other hand, the variables $(X_1 | X_1 + \tilde X_2 = g)$, $(X_2 | X_2 + \tilde X_1 = g')$ will be uniform on $A_1 \cap (A_2 + g)$ and $A_2 \cap (A_1+g')$ respectively, and such sets will typically (when non-empty) resemble random subsets of a coset of $H$, of density $1/m^2 \approx 1/K$. For these variables we also have $\dist{X'_1}{X'_2} \approx k$, and so we do not obtain~\eqref{d-opt} with the choice~\eqref{fiber-1}. The case of~\eqref{fiber-2} is broadly similar.

We therefore proceed to the endgame and consider the three variables in~\eqref{three-variables}.
Since $X_1+X_2$ is roughly uniform on $B \coloneqq \bigcup_{i,j} (x_i+y_j+H)$, a similar analysis to the second example shows that $(X_1+X_2 | X_1+X_2+\tilde X_1 + \tilde X_2 = s)$ is roughly uniform on $B \cap (B+s)$, which is typically a union of four cosets of $H$, and similarly for $(\tilde X_1+X_2 | X_1+X_2+\tilde X_1 + \tilde X_2 = s)$.
Moreover, one can check that these two variables are `50\% independent' (in that knowing one of them narrows the choice of the other to two cosets of $H$).

If we then apply \Cref{lem-bsg}, we find variables $X'_1$, $X'_2$ of the form given in~\eqref{from-bsg}.
In other words, $X'_1$ is a sum of two variables of the form $(X_1 | X_1 + \tilde X_2 = a)$ and $(X_2 | X_2 + \tilde X_1 = b)$.  Each will be uniform on a random subset of a coset of $H$, and adding them together will give something close to uniform on a coset of $H$.  The case of $X'_2$ is similar, so we get $\dist{X'_1}{X'_2} \approx 0$ and~\eqref{d-opt} follows.
(Again, permuting the variables may give slighly different conclusions.)

Finally, we consider a further example over $G = \Z$ instead of $\F_2^n$.
Take $X_1,X_2$ to have a `discrete Gaussian' distribution $p(x) = C_0 e^{-x^2/2r^2}$ with some large `width' $r$ and a normalizing constant $C_0$.
By~\cite[Theorem~1.13]{tao-entropy}, the distance $\dist{X_1}{X_2} \approx \frac12 \log 2$ is roughly minimal among random variables on $\Z$ with large entropy, and hence \Cref{de-prop} cannot hold over $\Z$ in the form stated.

Concretely, one can show that sums $X_i+\tilde X_j$ as in~\eqref{sum-1},~\eqref{sum-2} are (roughly) discrete Gaussians of width $\sqrt2r$, and fibres $(X_i | X_i+\tilde X_j)$ as in~\eqref{fiber-1},~\eqref{fiber-2} are (shifted) discrete Gaussians of width about $r/\sqrt2$. Hence, all these have essentially the same value of $\dist{-}{-}$, and there is no way of `improving' these variables with any of the choices considered in \Cref{first,second}.
One could also see directly that the conditional near-independence statements such as
\[
  \I{X_1+X_2 : X_1 + \tilde X_1 | X_1+X_2+\tilde X_1 + \tilde X_2} \approx 0
\]
do indeed hold for discrete Gaussians.

In the above one-dimensional example, the entropic doubling constant is bounded, and so this example does not, by itself, present a strong obstruction to the entire approach.  However, by considering discrete gaussian examples in a high-dimensional lattice $\Z^d$, the entropic doubling constant is now close to $\frac{d}{2} \log 2$, which is unbounded, and all the moves considered previously do not significantly reduce this constant.

Even though discrete Gaussians over $\F_2$ do not exist, this example places significant constraints on what we can hope to prove in general.  In particular, it shows that the endgame argument must `see' the finite characteristic in an essential way.

\subsubsection*{The contrapositive formulation}
The above discussion and the motivating examples which followed it were framed in terms of proving \Cref{de-prop} directly, by making various choices of $X'_1$ and $X'_2$, and the authors found this most natural when thinking about the problem.

However, when it comes to recording the proofs, it turns out to be notationally simpler to argue in the contrapositive.
Thus, suppose henceforth that we have a pair $(X_1, X_2)$ of $G$-valued random variables with $\dist{X_1}{X_2} = k$ for some $k$, and suppose that
\begin{equation}
  \label{to-contradict'}
  \tau[X'_1;X'_2] \geq \tau[X_1;X_2]
\end{equation} for every pair of $G$-valued random variables $(X'_1, X'_2)$.  Using the definition of $\tau$ (see~\eqref{tau-def}) we may rewrite~\eqref{to-contradict'} as
\begin{equation}\label{to-contradict}
  \dist{X'_1}{X'_2} \geq k - \costa{X'_1} - \costb{X'_2}.
\end{equation}

As in the discussion above, the idea now is to test~\eqref{to-contradict} with various choices of $(X'_1, X'_2)$ generated from $(X_1, X_2)$, the aim being to deduce from this that $k = 0$.

Testing~\eqref{to-contradict} with the choices~\eqref{sum-1},~\eqref{fiber-1} will lead us to the conclusion that
\begin{equation}\label{first-est}
  I_1 \coloneqq \bigI{ X_1+X_2 : \tilde X_1 + X_2 | X_1+X_2+\tilde X_1+\tilde X_2 } \leq 2 \eta k
\end{equation}
(that is, a more precise version of~\eqref{first-est-a}). The details of this deduction are provided in \Cref{first}. Testing~\eqref{to-contradict} with the choices~\eqref{sum-2},~\eqref{fiber-2} leads to
\begin{equation}\label{second-est}
  I_2 \coloneqq \bigI{ X_1+X_2 : X_1 + \tilde X_1 | X_1+X_2+\tilde X_1+\tilde X_2 } \leq 2 \eta k + \frac{2 \eta (2 \eta k - I_1)}{1 - \eta}
  \end{equation}
(a more precise version of~\eqref{second-est-a}), and of course the same estimate holds for $I_3$, since in fact $I_2=I_3$. The details of this deduction are provided in \Cref{second}.

  Finally, armed with the knowledge that~\eqref{first-est},~\eqref{second-est} hold, we proceed to the endgame, using choices of $X'_1,X_2'$ arising from~\eqref{three-variables} and entropic Balog--Szemer\'edi--Gowers (i.e., of forms such as~\eqref{from-bsg}) to finally show that $k = 0$.
The details are given in \Cref{endgame}.

The main reason for arguing by contradiction is that it greatly simplifies our discussions of conditioned random variables.
Note that~\eqref{to-contradict} implies a conditioned version of itself:
for any random variables $(X_1',Y_1)$ and $(X_2',Y_2)$ we may apply~\eqref{to-contradict} with $X'_1, X'_2$ replaced by each of the conditioned random variables $(X'_1 | Y_1 = y_1) , (X'_2 | Y_2 = y_2)$, obtaining
\begin{align*}
  & \dist{(X'_1  | Y_1 = y_1)}{(X'_2 | Y_2 = y_2)} \\
  & \qquad\qquad\qquad \geq k \begin{aligned}[t] &- \costa{(X'_1 | Y_1 = y_1)} \\
&- \costb{(X'_2 | Y_2 = y_2)}. \end{aligned}
  \end{align*}
Multiplying by $p_{Y_1}(y_1) p_{Y_2}(y_2)$ and summing, we have
\begin{align}
  \nonumber
  \dist{X'_1 | Y_1}{X'_2 | Y_2}  \geq k  &- \costa{X'_1 | Y_1} \\
  \label{to-contradict-avg}
  &- \costb{X'_2 | Y_2}.
\end{align}

More discussion of the notion of conditional distance may be found at~\eqref{cond-dist-def}.

To make the same argument in the direct (non-contrapositive) direction we would pick some particular $(y_1,y_2)$ such that
\begin{align*}
  \dist{(X_1' | Y_1=y_1)}{(X_2' | Y_2=y_2)} &+ \costa{(X'_1 | Y_1 = y_1)} \\
                                            &+ \costb{(X'_2 | Y_2 = y_2)}
\end{align*}
is at most the (weighted) average value of the same quantity over all $(y_1,y_2)$.
If we have to perform many such steps in sequence, this becomes notationally very taxing.
However, this is purely a matter of preference: for instance, in \Cref{lem:abstract} and its proof we instead choose to argue directly.

\subsubsection*{Remarks on odd $p$}
We will handle the case $p$ odd in our forthcoming paper~\cite{ggmt-pfr-odd}.
All of the above steps go through, except for the final contradiction in the `endgame' where we used an argument specific to characteristic $2$.
To get a contradiction in characteristic $p$, it is necessary to run a variant of the argument with a $p$-partite `distance' function $\D[X_1;\dots; X_p]$ replacing $\dist{X_1}{X_2}$ as the main term in the functional to be minimized.
This introduces additional notational complexity: for instance we now have an array $(X_{ij})_{i,j \in \F_p}$ of random variables, and the fibring lemma must be generalized to this $p$-partite distance and applied $p-1$ times.

\section{Fibring lemma}%
\label{sec4}
Here we record the fibring lemma, that is to say~\cite[Proposition 1.4]{gmt}, with an explicit error term.  (The latter was mentioned in~\cite{gmt}, but only as a casual remark.)

 \begin{proposition}\label{projections-1}
   Let $\pi \colon H \to H'$ be a homomorphism between abelian groups and let $Z_1,Z_2$ be $H$-valued random variables. Then we have
   \[
     \dist{Z_1}{Z_2} \geq \dist{\pi(Z_1)}{\pi(Z_2)} + \dist{Z_1|\pi(Z_1)}{Z_2 |\pi(Z_2)}
   \]
   where the notation is as in~\eqref{cond-dist-def}.
   Moreover, if $Z_1,Z_2$ are taken to be independent, then the difference between the two sides is
\begin{equation}
\label{mut-info-cond} \bigI{ Z_1 - Z_2 : (\pi(Z_1), \pi(Z_2)) \; | \; \pi(Z_1 - Z_2)}.
\end{equation} \end{proposition}
\begin{proof}
Let $Z_1,Z_2$ be independent throughout.  We have
\begin{align*}
  & \bigdist{Z_1  |\pi(Z_1)}{ Z_2 |\pi(Z_2)} \\
  & = \bigH{ Z_1 - Z_2 | \pi(Z_1),\pi(Z_2)} - \tfrac{1}{2} \ent{Z_1 | \pi(Z_1)} - \tfrac{1}{2}\ent{Z_2 | \pi(Z_2)} \\
  & \leq  \bigH{ Z_1 - Z_2 | \pi(Z_1-Z_2)}  - \tfrac{1}{2} \ent{Z_1 | \pi(Z_1)} - \tfrac{1}{2}\ent{Z_2 | \pi(Z_2)} \\
  & = \dist{Z_1}{Z_2} - \dist{\pi(Z_1)}{\pi(Z_2)}.
\end{align*}
In the middle step, we used submodularity of entropy, and in the last step we used the fact that
\[ \bigH{ Z_1 - Z_2 |  \pi(Z_1-Z_2)} = \ent{Z_1 - Z_2} - \ent{\pi(Z_1-Z_2)}\] (since $Z_1 - Z_2$ determines $\pi(Z_1 - Z_2)$) and that
\[ \ent{Z_i| \pi(Z_i)} = \ent{Z_i} - \ent{\pi(Z_i)}\] (since $Z_i$ determines $\pi(Z_i)$).
This gives the claimed inequality. The difference between the two sides is precisely
\[ \bigH{Z_1 - Z_2  | \pi(Z_1 - Z_2)} - \bigH{Z_1 - Z_2  | \pi(Z_1),\pi(Z_2)}.\]
To rewrite this in terms of (conditional) mutual information, we use the identity
\[ \ent{A|B} - \ent{A | B,C} = \I{A : C | B},\] taking
$A \coloneqq Z_1 - Z_2$, $B \coloneqq \pi(Z_1 - Z_2)$ and $C \coloneqq (\pi(Z_1),\pi(Z_{2}))$, and noting that in this case $\ent{A | B,C} = \ent{A | C}$ since $C$ uniquely determines $B$.
This completes the proof.
\end{proof}

We extract the specific special case of this result that we will need in our arguments.

\begin{corollary}\label{cor-fibre}
   Let $Y_1,Y_2,Y_3$ and $Y_4$ be independent random variables taking values in some abelian group $G$.
  Then
\begin{align*}
 \nonumber
  & \dist{Y_1-Y_3}{Y_2-Y_4} + \bigdist{Y_1|Y_1-Y_3}{Y_2|Y_2-Y_4} \\
  &\qquad + \bigI{ Y_1-Y_2 : Y_2 - Y_4 | Y_1-Y_2-Y_3+Y_4 } = \dist{Y_1}{Y_2} + \dist{Y_3}{Y_4}.
\end{align*}
\end{corollary}
\begin{proof}  We apply \Cref{projections-1} with $H \coloneqq G \times G$, $H' \coloneqq G$, $\pi$ the subtraction homomorphism $\pi(x,y) \coloneqq x-y$, and with the random variables $Z_1 \coloneqq (Y_1,Y_3)$ and $Z_2 \coloneqq (Y_2,Y_4)$. Then by independence we easily calculate
\[
  \dist{Z_1}{Z_2} = \dist{Y_1}{Y_2} + \dist{Y_3}{Y_4}
\]
while by definition
\[
  \dist{\pi(Z_1)}{\pi(Z_2)} = \dist{Y_1-Y_3}{Y_2-Y_4}.
\]
Furthermore,
\[
  \dist{Z_1|\pi(Z_1)}{Z_2|\pi(Z_2)} = \dist{Y_1|Y_1-Y_3}{Y_2|Y_2-Y_4},
\]
since $Z_1=(Y_1,Y_3)$ and $Y_1$ are linked by an invertible affine transformation once $\pi(Z_1)=Y_1-Y_3$ is fixed, and similarly for $Z_2$ and $Y_2$.
Finally, we have
\begin{align*}
  &\bigI{Z_1 - Z_2 : (\pi(Z_1),\pi(Z_2)) \,|\, \pi(Z_1) + \pi(Z_2)} \\
  &\ = \bigI{(Y_1-Y_2, Y_3-Y_4) : (Y_1-Y_3, Y_2-Y_4) \,|\, Y_1-Y_2-Y_3+Y_4} \\
  &\ = \bigI{Y_1-Y_2 : Y_2-Y_4 \,|\, Y_1-Y_2-Y_3+Y_4}
\end{align*}
where in the last line we used the fact that $(Y_1-Y_2, Y_1-Y_2-Y_3+Y_4)$ uniquely determine $Y_3-Y_4$ and similarly
$(Y_2-Y_4, Y_1-Y_2-Y_3+Y_4)$ uniquely determine $Y_1-Y_3$.
\end{proof}
\begin{remark}\label{rem:char2}
  Of course, in our main characteristic $2$ setting, the minus signs in \Cref{projections-1,cor-fibre} may be replaced with plus signs. \end{remark}

\section{First estimate}\label{first}

Recall that in this and subsequent sections we are working on the assumption that~\eqref{to-contradict} and its conditioned variant~\eqref{to-contradict-avg} hold, aiming to prove that $k = 0$ and hence conclude (the contrapositive of) Proposition~\ref{de-prop}.

Recall also that $X_1, X_2, \tilde X_1, \tilde X_2$ are independent random variables, with $X_1,\tilde X_1$ copies of $X_1$ and $X_2,\tilde X_2$ copies of $X_2$.

In this section we establish the upper bound~\eqref{first-est}, which was that
\[
  I_1 \coloneqq \bigI{ X_1+X_2 : \tilde X_1 + X_2 | X_1+X_2+\tilde X_1+\tilde X_2 } \leq 2 \eta k.
\]
We apply \Cref{cor-fibre} (and Remark \ref{rem:char2}) with the choice
\[
  (Y_1,Y_2,Y_3,Y_4) \coloneqq (X_1, X_2, \tilde X_2, \tilde X_1).
\]
It gives
\begin{align}
  \nonumber
&\dist{X_1+\tilde X_2}{X_2+\tilde X_1} + \bigdist{X_1|X_1+\tilde X_2}{X_2|X_2+\tilde X_1} \\
\label{second-main}
&\quad + \bigI{ X_1+ X_2 : \tilde X_1 + X_2 \,|\, X_1 + X_2 + \tilde X_1 + \tilde X_2 } = 2k,
\end{align}
since $\dist{X_1}{X_2} = k$.
Applying~\eqref{to-contradict},~\eqref{to-contradict-avg}, we have
\begin{align*}
  \dist{X_1+\tilde X_2}{X_2+\tilde X_1} \geq k &- \costa{X_1+\tilde X_2} \\& \qquad- \costb{X_2+\tilde X_1}
\end{align*}
and
\begin{align}
  \nonumber
  & \dist{X_1|X_1+\tilde X_2}{X_2|X_2+\tilde X_1}  \\ \nonumber & \qquad\quad \geq k - \costa{X_1 | X_1 + \tilde X_2} \\
  & \qquad\qquad\qquad\qquad  - \costb{X_2 | X_2 + \tilde X_1}.
   \label{second-tc2}
\end{align}
It therefore suffices to prove that
\begin{align}
\nonumber
&(\dist{X_1^0}{X_1+\tilde X_2}-\dist{X_1^0}{X_1})+ (\dist{X_2^0}{X_2+\tilde X_1}-\dist{X_2^0}{X_2})\\
&\qquad \qquad +(\dist{X_1^0}{X_1|X_1+\tilde X_2}-\dist{X_1^0}{X_1}) \nonumber \\ & \qquad \qquad \qquad \qquad  +(\dist{X_2^0}{X_2|X_2+\tilde X_1}-\dist{X_2^0}{X_2}) \leq 2k.
\label{suff-for-first-est}
\end{align}

We pause to state some lemmas which we will use for bounding a number of similar expressions, both here and in the next two sections.

The first is a bound relating conditioned and unconditioned variants of the Ruzsa distance.

\begin{lemma}%
  \label{cond-dist-fact}
  Suppose that $(X, Z)$ and $(Y, W)$ are random variables, where $X, Y$ take values in some abelian group. Then
  \[    \dist{X  | Z}{Y | W} \leq \dist{X}{Y} + \tfrac{1}{2} \I{X : Z} + \tfrac{1}{2} \I{Y : W}.\]
\end{lemma}
\begin{proof}
The definition of conditional distance is given in~\eqref{cond-dist-def}.

Using the alternative expression~\eqref{cond-dist-alt}, if $(X',Z'), (Y',W')$ are independent copies of the variables $(X,Z)$, $(Y,W)$, we have
\begin{align*}
  \dist{X  | Z}{Y | W} &= \ent{X'-Y'|Z',W'} - \tfrac{1}{2} \ent{X'|Z'} - \tfrac{1}{2}\ent{Y'|W'} \\
                       &\le \ent{X'-Y'}- \tfrac{1}{2} \ent{X'|Z'} - \tfrac{1}{2}\ent{Y'|W'} \\
                       &= \dist{X'}{Y'} + \tfrac{1}{2} \I{X' : Z'} + \tfrac{1}{2} \I{Y' : W'}.
\end{align*}
Here, in the middle step we used~\eqref{cond-dec}, and in the last step we used the definitions of $\dist{-}{-}$ and $\I{-}$.
\end{proof}

In the proof of the next lemma we will use the fact that, for any independent random variables $X, Y, Z$ taking values in an abelian group, we have the inequality
\begin{equation}\label{kv-2}
  \ent{X + Y + Z} - \ent{X + Y} \leq \ent{Y+Z} - \ent{Y}.
\end{equation}
This is a result of Madiman~\cite[Theorem I]{madiman}, and is a more general form of an inequality of Kaimanovich and Vershik~\cite[Proposition 1.3]{kv}. It can be viewed as an entropy analogue of an inequality of Pl\"unnecke~\cite{plunnecke}. For the convenience of the reader, we give the proof in \Cref{entropy-app}.

\begin{lemma}\label{first-useful}
  Let $X, Y, Z$ be random variables taking values in some abelian group, and with $Y, Z$ independent. Then we have\footnote{We thank Floris van Doorn for noting a sign discrepancy in a previous version of this statement, which was uncovered as part of the effort to formalize the results of this paper in {\tt Lean 4}.}
\begin{align}\nonumber \dist{X}{Y - Z} - \dist{X}{Y} &  \leq \tfrac{1}{2} (\ent{Y-Z} - \ent{Y}) \\ & = \tfrac{1}{2} \dist{Y}{Z} + \tfrac{1}{4} \ent{Z} - \tfrac{1}{4} \ent{Y} \label{lem51-a} \end{align}
and
\begin{align}\nonumber
\dist{X}{Y|Y-Z} - \dist{X}{Y} & \leq \tfrac{1}{2} \bigl(\ent{Y-Z} - \ent{Z}\bigr) \\ & = \tfrac{1}{2} \dist{Y}{Z} + \tfrac{1}{4} \ent{Y} - \tfrac{1}{4} \ent{Z}.
  \label{ruzsa-3}
\end{align}
\end{lemma}
\begin{proof}
We first prove~\eqref{lem51-a}. We may assume (taking an independent copy) that $X$ is independent of $Y, Z$. Then we have
\begin{align*}  \dist{X}{Y-Z} & - \dist{X}{Y} \\ & = \ent{X - Y + Z} - \ent{X-Y} - \tfrac{1}{2}\ent{Y - Z} + \tfrac{1}{2} \ent{Y}.\end{align*}
Combining this with~\eqref{kv-2} (applied with $Y$ replaced by $-Y$) gives the required bound. The second form of the result is immediate from the definition of $\dist{Y}{Z}$, since $Y, Z$ are independent.

Turning to~\eqref{ruzsa-3}, we have
\begin{align*} \I{Y : Y-Z} & = \ent{Y} + \ent{Y - Z} - \ent{Y, Y - Z} \\ & = \ent{Y} + \ent{Y - Z} - \ent{Y, Z}  = \ent{Y - Z} - \ent{Z},\end{align*}
and so~\eqref{ruzsa-3} is a consequence of~\Cref{cond-dist-fact}.
Once again the second form of the result is immediate from the definition of $\dist{Y}{Z}$.
\end{proof}

We return to our main task of establishing~\eqref{suff-for-first-est}, and hence~\eqref{first-est}.

As usual, in the main setting $G=\F_2^n$ we may replace all minus signs with plus signs in the statement of \Cref{first-useful}.
Hence by \Cref{first-useful} (and recalling that $k$ is defined to be $\dist{X_1}{X_2}$) we have
\[ \dist{X^0_1}{X_1+\tilde X_2} - \dist{X^0_1}{X_1} \leq \tfrac{1}{2} k + \tfrac{1}{4} \ent{X_2} - \tfrac{1}{4} \ent{X_1},\] \[  \dist{X^0_2}{X_2+\tilde X_1} - \dist{X^0_2}{X_2} \leq \tfrac{1}{2} k + \tfrac{1}{4} \ent{X_1} - \tfrac{1}{4} \ent{X_2},\]
\begin{equation}  \dist{X_1^0}{X_1|X_1+\tilde X_2} - \dist{X_1^0}{X_1} \leq \tfrac{1}{2} k + \tfrac{1}{4} \ent{X_1} - \tfrac{1}{4} \ent{X_2}
                                 \label{second-dist1}
\end{equation}
and
\begin{equation}
  \label{second-dist2}
  \dist{X_2^0}{X_2|X_2+\tilde X_1} - \dist{X_2^0}{X_2} \leq \tfrac{1}{2}k + \tfrac{1}{4} \ent{X_2} - \tfrac{1}{4} \ent{X_1}.
\end{equation}
Adding all these inequalities, we obtain~\eqref{suff-for-first-est}.

For use in the next two sections, we note that subtracting~\eqref{second-tc2} from~\eqref{second-main}, and combining the resulting inequality with~\eqref{second-dist1} and~\eqref{second-dist2} gives the bound
\[
  \dist{X_1+\tilde X_2}{X_2+\tilde X_1} \le (1 + \eta) k - I_1,
\]
which is equivalent to
\begin{equation}
  \label{HS-bound}
  \bigH{X_1+X_2+\tilde X_1+\tilde X_2} \le \tfrac{1}{2} \ent{X_1}+\tfrac{1}{2} \ent{X_2} + (2 + \eta) k - I_1.
\end{equation}
One could also bound the left-hand side using~\eqref{kv-2} (twice), but by making use of our hypothesis~\eqref{to-contradict-avg} as above we obtain a slightly better constant.

\section{Second estimate}\label{second}

In this section we establish the upper bound~\eqref{second-est}, which was
\[
  I_2 \coloneqq \bigI{ X_1+X_2 : X_1 + \tilde X_1 | X_1+X_2+\tilde X_1+\tilde X_2 } \leq 2 \eta k + \frac{2 \eta (2 \eta k - I_1)}{1 - \eta}.
\]
We apply \Cref{cor-fibre} (and Remark \ref{rem:char2}), but now with the choice
\[
  (Y_1,Y_2,Y_3,Y_4) \coloneqq (X_2, X_1, \tilde X_2, \tilde X_1).
\]
Now \Cref{cor-fibre} can be rewritten as
\begin{align*}
  &\dist{X_1+\tilde X_1}{X_2+\tilde X_2} + \bigdist{X_1|X_1+\tilde X_1}{X_2|X_2+\tilde X_2} \\
  &\quad + \bigI{ X_1+X_2 : X_1 + \tilde X_1 \,|\, X_1+X_2+\tilde X_1+\tilde X_2 } = 2k,
\end{align*}
recalling once again that $k \coloneqq \dist{X_1}{X_2}$.  From~\eqref{to-contradict} and~\eqref{to-contradict-avg} as before, one has
\begin{align}
  \nonumber
  \dist{X_1+\tilde X_1}{X_2+\tilde X_2} \geq k &- \costa{X_1+\tilde X_1} \\ &- \costb{X_2+\tilde X_2}
\label{1122}
\end{align}
and
\begin{align*}
  \dist{X_1|X_1+\tilde X_1}{X_2|X_2+\tilde X_2}   \geq k &- \costa{X_1|X_1+\tilde X_1} \\& - \costb{X_2|X_2+\tilde X_2} .
\end{align*}
Now \Cref{first-useful} gives
\begin{equation}\label{1122a} \dist{X^0_1}{X_1+\tilde X_1} - \dist{X^0_1}{X_1} \leq \tfrac{1}{2} \dist{X_1}{X_1},\end{equation}
\begin{equation}\label{1122b}
  \dist{X^0_2}{X_2+\tilde X_2} - \dist{X^0_2}{X_2} \leq \tfrac{1}{2} \dist{X_2}{X_2},
\end{equation}
\[
  \dist{X^0_1}{X_1|X_1+\tilde X_1} -  \dist{X^0_1}{X_1} \leq  \tfrac{1}{2} \dist{X_1}{X_1},
\]
and
\[
  \dist{X^0_2}{X_2|X_2+\tilde X_2} -  \dist{X^0_2}{X_2} \leq \tfrac{1}{2} \dist{X_2}{X_2}.
  \]
Combining all these inequalities and cancelling terms, we obtain
\begin{equation}\label{combined}
\I{ X_1+X_2 : X_1 + \tilde X_1 | X_1+X_2+\tilde X_1+\tilde X_2 } \leq \eta ( \dist{X_1}{X_1} + \dist{X_2}{X_2} ).
\end{equation}
One could bound the right-hand side by $4\eta k$ using the Ruzsa triangle inequality, but a more efficient approach is as follows.  First, by combining~\eqref{1122},~\eqref{1122a} and~\eqref{1122b}, we obtain
\begin{equation}\label{d12}
  \dist{X_1+\tilde X_1}{X_2+\tilde X_2} \geq k - \frac{\eta}{2} ( \dist{X_1}{X_1} + \dist{X_2}{X_2} ).
\end{equation}
We may also expand
\begin{align*}
 & \dist{X_1+\tilde X_1}{X_2+\tilde X_2} \\ &= \ent{ X_1+\tilde X_1 + X_2 + \tilde X_2}  - \tfrac{1}{2} \ent{X_1+\tilde X_1} - \tfrac{1}{2} \ent{X_2+\tilde X_2} \\
  &= \ent{ X_1+\tilde X_1 + X_2 + \tilde X_2}  - \tfrac{1}{2} \ent{X_1} - \tfrac{1}{2} \ent{X_2}  \\ & \qquad\qquad\qquad   - \tfrac{1}{2} \left( \dist{X_1}{X_1} + \dist{X_2}{X_2} \right),
\end{align*}
and hence by~\eqref{HS-bound}
\[
  \dist{X_1+\tilde X_1}{X_2+\tilde X_2} \leq (2+\eta) k - \tfrac{1}{2} \left( \dist{X_1}{X_1} + \dist{X_2}{X_2} \right) - I_1.
\]
Combining this bound with~\eqref{d12} we obtain
\begin{equation}\label{x12}
  \dist{X_1}{X_1} + \dist{X_2}{X_2} \leq 2 k + \frac{2(2 \eta k - I_1)}{1-\eta}.
\end{equation}
Therefore by~\eqref{combined} we have the desired bound
\[
\I{ X_1+X_2 : X_1 + \tilde X_1 | X_1+X_2+\tilde X_1+\tilde X_2 } \leq 2 \eta k + \frac{2\eta(2\eta k - I_1)}{1-\eta}.
\]

\section{Endgame}\label{endgame}

In this section we conclude the proof of \Cref{de-prop}.  Let us begin by recording an inequality which will be used several times in the calculations below.
\begin{lemma}\label{second-useful}
Let $X, Y, Z, Z'$ be random variables taking values in some abelian group, and with $Y, Z, Z'$ independent. Then we have
\begin{align}\nonumber
& \dist{X}{Y - Z | Y - Z - Z'} - \dist{X}{Y} \\ & \qquad \leq \tfrac{1}{2} ( \ent{Y - Z - Z'} + \ent{Y - Z} - \ent{Y} - \ent{Z'}).\label{7111}
\end{align}
\end{lemma}
\begin{proof}
By~\eqref{ruzsa-3} (with a change of variables) we have
\[ \dist{X}{Y - Z | Y - Z - Z'} - \dist{X}{Y - Z} \leq \tfrac{1}{2}( \ent{Y - Z - Z'} - \ent{Z'}).\]
Adding this to~\eqref{lem51-a} gives the result.
\end{proof}

Turning now to the main argument, let $X_1,X_2,\tilde X_1,\tilde X_2$ be as before, and introduce the random variables
\[ U \coloneqq X_1 + X_2, \qquad V \coloneqq \tilde X_1 + X_2, \qquad W \coloneqq X_1 + \tilde X_1\] and
\[   S \coloneqq X_1 + X_2 + \tilde X_1 + \tilde X_2.\]
From the definitions~\eqref{first-est-a},~\eqref{second-est-a},~\eqref{third-est-a} of $I_1, I_2, I_3$ and the above notation, we see that
\[
  I_1 = \I{U : V \, | \, S}, \qquad I_2 = \I{W : U \, | \, S}, \qquad I_3 = \I{V : W \, | \,S}.
\]
Recall from~\eqref{first-est} that $I_1 \leq 2 \eta k$. From~\eqref{second-est} (and since $I_2 = I_3$) we have the inequalities
\[   \I{V : W \, | \,S} , \;  \I{W : U \, | \, S} \leq 2 \eta k + \frac{2\eta(2 \eta k - I_1)}{1-\eta} .
\]
Summing these two inequalities and the equality $\I{U : V \, | \, S} = I_1$ gives
\begin{align}
  \nonumber
   \label{uvw-s} \I{U : V \, | \, S} + \I{V : W \, | \,S} &+ \I{W : U \, | \, S}\\
   \nonumber
   &\leq I_1+4\eta k+ \frac{4\eta(2 \eta k - I_1)}{1-\eta}\\
  &= 6 \eta k - \frac{1 - 5 \eta}{1-\eta} (2 \eta k - I_1).
\end{align}
We encourage the reader to read the argument assuming first that $I_1 = 2 \eta k$, in which case the calculations are cleaner.

We assemble some preliminary estimates on distances. By \Cref{second-useful} (again replacing all minus signs with plus signs, and taking $X = X_1^0$, $Y = X_1$, $Z = X_2$ and $Z' = \tilde X_1 + \tilde X_2$, so that $Y + Z = U$ and $Y + Z + Z' = S$) we have, noting that $\ent{Y+Z} = \ent{Z'}$,
\[
  \dist{X^0_1}{U|S} - \dist{X^0_1}{X_1} \leq \tfrac{1}{2} (\ent{S} -  \ent{X_1}).
\]
Further applications of \Cref{second-useful} give
\begin{align*}
\dist{X^0_2}{U|S} - \dist{X^0_2}{X_2} &\leq \tfrac{1}{2} (\ent{S} -  \ent{X_2}) \\
\dist{X^0_1}{V|S} - \dist{X^0_1}{X_1} &\leq \tfrac{1}{2} (\ent{S} -  \ent{X_1})\\
\dist{X^0_2}{V|S} - \dist{X^0_2}{X_2} &\leq \tfrac{1}{2} (\ent{S} -  \ent{X_2})
\end{align*}
and
\[ \dist{X^0_1}{W|S} - \dist{X^0_1}{X_1} \leq \tfrac{1}{2} (\ent{S} + \ent{W} - \ent{X_1} - \ent{W'}),\] where $W' := X_2 + \tilde X_2$.
To treat $\dist{X^0_2}{W|S}$, first note that this equals $\dist{X^0_2}{W'|S}$, since for a fixed choice $s$ of $S$ we have $W' = W + s$. Now we may apply \Cref{second-useful} to obtain
\[ \dist{X^0_2}{W'|S} - \dist{X^0_2}{X_2} \leq \tfrac{1}{2} (\ent{S} + \ent{W'} - \ent{X_2} - \ent{W}).\]
Summing these six estimates and using~\eqref{HS-bound}, we conclude that
\begin{align}
  \nonumber
  \sum_{i=1}^2 \sum_{A\in\{U,V,W\}} \big(\dist{X^0_i}{A|S} & - \dist{X^0_i}{X_i}\big) \\
    &\leq 3\ent{S} - \tfrac{3}{2} \ent{X_1} - \tfrac{3}{2}\ent{X_2} \nonumber \\
    &\leq (6 - 3\eta) k + 3(2 \eta k - I_1).
  \label{total-dist}
\end{align}
Now we come to the key observation we will exploit, which is that
\begin{equation}
\label{uvw-sum} U + V + W = 0.
\end{equation}
Here, of course, we are using the fact that we are in characteristic $2$.
This is the only critical use of this fact in the argument, in the sense that it cannot be avoided by judicious insertion of negative signs or by accepting slightly worse constants by invoking the inequality $\dist{X}{-Y} \leq 3 \dist{X}{Y}$ (see~\cite[Equation 17]{tao-entropy}) and related estimates.

To see the force of~\eqref{uvw-s} and~\eqref{uvw-sum}, we state the following general claim.
\begin{lemma}%
  \label{lem:abstract}
  Let $G=\F_2^n$ and let $(T_1,T_2,T_3)$ be a $G^3$-valued random variable such that $T_1+T_2+T_3=0$ holds identically. Set
  \begin{equation}\label{delta-t1t2t3-def}
    \delta \coloneqq \sum_{1 \leq i < j \leq 3} \I{T_i;T_j}.
  \end{equation}
  Then there exist random variables $T'_1, T'_2$ such that
  \begin{align*}  \dist{T'_1}{T'_2} + & \eta (\dist{X_1^0}{T'_1} - \dist{X_1^0}{X_1}) + \eta(\dist{X_2^0}{T'_2} - \dist{X_2^0}{X_2}) \\ & \leq  \delta + \frac{\eta}{3} \biggl( \delta + \sum_{i=1}^2 \sum_{j = 1}^3 (\dist{X^0_i}{T_j} - \dist{X^0_i}{X_i}) \biggr).
  \end{align*}
\end{lemma}

We recommend that the reader first work through this proof in the case $\delta = 0$, when the variables $T_1, T_2, T_3$ are independent.  Recalling~\eqref{eq:i-eq-0}, the proof below collapses to taking $T'_1,T'_2$ to be some two of $\{T_1,T_2,T_3\}$, without the need to apply \Cref{lem-bsg}.

\begin{proof}
We apply the variant of the entropic Balog--Szemer\'edi--Gowers theorem stated in \Cref{lem-bsg}, taking $(A,B) = (T_1, T_2)$ there.
Since $T_1 + T_2 = T_3$, the conclusion is that
\begin{align} \nonumber \sum_{t_3}p_{T_3}(t_3) & \dist{(T_1 | T_3 = t_3)}{(T_2 | T_3 = t_3)} \\ & \leq 3 \I{T_1 : T_2} + 2 \ent{T_3} - \ent{T_1} - \ent{T_2}.\label{bsg-t1t2}\end{align}
The right-hand side in~\eqref{bsg-t1t2} can be rearranged as
\begin{align*} & 2( \ent{T_1} + \ent{T_2} + \ent{T_3}) - 3 \ent{T_1,T_2} \\ & = 2(\ent{T_1} + \ent{T_2} + \ent{T_3}) - \ent{T_1,T_2} - \ent{T_2,T_3} - \ent{T_1, T_3} = \delta,\end{align*}
using the fact that all three terms $\ent{T_i,T_j}$ are equal to $\ent{T_1,T_2,T_3}$ and hence to each other.
We also have
\begin{align*}
&  \sum_{t_3}p_{T_3}(t_3) \bigl(\dist{X^0_1}{(T_1 | T_3=t_3)} - \dist{X^0_1}{X_1}\bigr) \\
&\quad = \dist{X^0_1}{T_1 | T_3} - \dist{X^0_1}{X_1} \leq \dist{X^0_1}{T_1} - \dist{X^0_1}{X_1} + \tfrac{1}{2} \I{T_1 : T_3}
\end{align*}
by \Cref{cond-dist-fact}, and similarly
\begin{align*}
&  \sum_{t_3}p_{T_3}(t_3) (\dist{X^0_2}{(T_2 | T_3=t_3)} - \dist{X^0_2}{X_2}) \\
&\quad\quad\quad\quad\quad\quad \leq \dist{X^0_2}{T_2} - \dist{X^0_2}{X_2} + \tfrac{1}{2} \I{T_2 : T_3}.
\end{align*}
Temporarily define
\[ \psi[Y_1; Y_2] \coloneqq \dist{Y_1}{Y_2} +  \eta (\dist{X_1^0}{Y_1} - \dist{X_1^0}{X_1}) + \eta(\dist{X_2^0}{Y_2} - \dist{X_2^0}{X_2}).\]

Putting the above observations together, we have
\begin{align*}
 \sum_{t_3}p_{T_3}(t_3) \psi[(T_1 | T_3=t_3); (T_2 | T_3=t_3)] \leq \delta + \eta (\dist{X^0_1}{T_1}-\dist{X^0_1}{X_1}) \\
   + \eta (\dist{X^0_2}{T_2}-\dist{X^0_2}{X_2}) + \tfrac12 \eta \I{T_1:T_3} + \tfrac12 \eta \I{T_2:T_3}.
 \end{align*}
 Choosing some $t_3$ in the support of $T_3$ that minimizes the $\psi[-;-]$ value, and setting $T'_{1,3} \coloneqq (T_1 | T_3 = t_3)$, $T'_{2,3} \coloneqq (T_2 | T_3 = t_3)$, we have
\begin{align}\nonumber
 \psi[T_{1,3}';T_{2,3}'] \leq \delta + \eta (& \dist{X^0_1}{T_1}-\dist{X^0_1}{X_1})
    + \eta (\dist{X^0_2}{T_2}-\dist{X^0_2}{X_2}) \\ & + \tfrac12 \eta \I{T_1:T_3} + \tfrac12 \eta \I{T_2:T_3}.
   \label{eq:t1323}
\end{align}
We now repeat this analysis for all permutations of $\{T_1,T_2,T_3\}$ to get variables $T'_{\alpha,\gamma}, T'_{\beta,\gamma}$ for $\{\alpha,\beta,\gamma\}$ ranging over all six permutations of $\{1,2,3\}$.
Averaging the resulting inequalities~\eqref{eq:t1323}, and recalling the definition~\eqref{delta-t1t2t3-def} of $\delta$, we get
\[ \tfrac16 \sum_{\alpha,\beta,\gamma} \psi[T_{\alpha,\gamma}';T_{\beta,\gamma}']  \leq  \delta + \frac{\eta}{3} \biggl( \delta + \sum_{i=1}^2 \sum_{j = 1}^3 (\dist{X^0_i}{T_j}-\dist{X^0_i}{X_i}) \biggr),
\]
from which the result follows (taking $T'_1,T'_2$ to be $T'_{\alpha,\gamma},T'_{\beta,\gamma}$ for some $(\alpha,\beta,\gamma)$ that gives at most the average value).
\end{proof}

Applying \Cref{lem:abstract} with any random variables $(T_1,T_2,T_3)$ such that $T_1+T_2+T_3=0$ holds identically, and applying~\eqref{to-contradict} with $X'_1 = T'_1$, $X'_2 = T'_2$, we deduce that
\[
  k \leq \delta + \frac{\eta}{3} \biggl( \delta + \sum_{i=1}^2 \sum_{j = 1}^3 (\dist{X^0_i}{T_j} -\dist{X^0_i}{X_i}) \biggr).
\]
Note that $\delta$ is still defined by~\eqref{delta-t1t2t3-def} and thus depends on $T_1,T_2,T_3$.
In particular we may apply this for
\[
  T_1 = (U | S = s),\qquad T_2 = (V | S = s), \qquad T_3 = (W | S = s)
\]
for $s$ in the range of $S$ (which is a valid choice by~\eqref{uvw-sum}) and then average over $s$ with weights $p_S(s)$, to obtain
\begin{equation}\label{almost-done}
  k \leq \tilde \delta + \frac{\eta}{3} \biggl( \tilde \delta + \sum_{i=1}^2 \sum_{A\in\{U,V,W\}} \bigl(  \dist{X^0_i}{A|S} - \dist{X^0_i}{X_i}\bigr) \biggr),
\end{equation}
where
\[
  \tilde \delta \coloneqq  \I{U : V | S} + \I{V : W | S} + \I{W : U | S}.
\]
Putting this together with~\eqref{uvw-s} and~\eqref{total-dist}, we conclude that
\begin{align*}
  k &\leq \Bigl(1+\frac{\eta}{3}\Bigr)\Bigl(6\eta k-\frac{1-5\eta}{1-\eta}(2\eta k-I_1)\Bigr)+\frac{\eta}{3}\Bigl((6-3\eta)k+3(2\eta k-I_1)\Bigr)\\
  &= (8\eta + \eta^2) k - \biggl( \frac{1 - 5 \eta}{1-\eta}\Bigl(1 + \frac{\eta}{3}\Bigr) -  \eta \biggr)(2 \eta k - I_1)\\
  &\leq (8 \eta + \eta^2) k
 \end{align*}
since the quantity $2 \eta k - I_1$ is non-negative (by~\eqref{first-est}), and its coefficient in the above expression is non-positive provided that $\eta(2\eta + 17) \le 3$, which is certainly the case for our choice $\eta = \frac{1}{9}$ (and in fact for any $\eta \leq \frac{1}{6}$).
Moreover, for\footnote{In fact we can take any $\eta<\frac{1}{4 + \sqrt{17}} = \frac{1}{8.1231\dots}$, and the other constants in the paper can be improved accordingly. In \cite{liao1,liao2} the arguments were modified to also apply in the regime $\eta < \frac{1}{8}$.} $\eta=\tfrac{1}{9}$ we have $8 \eta + \eta^2 < 1$. It follows that $k=0$, as desired.
The proof of \Cref{de-prop} (and hence of all our results) is complete.

\appendix

\section{Entropy and additive combinatorics}\label{entropy-app}

In this appendix we record some standard inequalities regarding Shannon entropy, as well as the standard `entropic Ruzsa calculus' concerning entropies of sums of random variables.

First we remark that if $X$ takes values in a set $S$ then, by Jensen's inequality,
\begin{equation}\label{convexity-bound} \ent{X} \leq \log |S|.\end{equation}
Also, denoting by $p_X$ the density function of $X$,
\[ \ent{X} = \sum_x p_X(x) \log \frac{1}{p_X(x)} \geq \min_{x : p_X(x)>0} \log \frac{1}{p_X(x)},\] and therefore
\begin{equation}\label{px-lower} \max_x p_X(x) \geq e^{-\ent{X}}.\end{equation}

Given a pair $(X,Y)$ of random variables, the \emph{conditional entropy} $\ent{X|Y}$ is defined by the formula
\[
  \ent{X|Y} \coloneqq \sum_y p_Y(y) \ent{X|Y=y}
\]
where $y$ ranges over the support of $p_Y$, and $X|Y=y$ denotes the random variable $X$ conditioned on the event $Y=y$.  We have the fundamental \emph{chain rule}
\begin{equation}
  \label{chain-rule}
  \ent{X,Y} = \ent{X|Y} + \ent{Y}.
\end{equation}
Here we abbreviate $\ent{(X,Y)}$ as $\ent{X,Y}$, and will make similar abbreviations regarding other information-theoretic quantities in this paper without further comment; for instance, $\ent{(X,Y)|(Z,W)}$ becomes $\ent{X,Y|Z,W}$.  Note that~\eqref{chain-rule} implies a conditional generalization
\[
  \ent{X,Y|Z} = \ent{X|Y,Z} + \ent{Y|Z}.
\]
for all random variables $X,Y,Z$.

The \emph{mutual information} $\I{X:Y}$ is defined by the formula
\[
  \begin{split}
    \I{X:Y} &= \ent{X} + \ent{Y} - \ent{X,Y} \\
            &= \ent{X} - \ent{X|Y} \\
            &= \ent{Y} - \ent{Y|X},
  \end{split}
\]
and is non-negative by a standard application of Jensen's inequality, vanishing precisely when $X,Y$ are independent; in particular
\begin{equation}
\label{indep}
  \ent{X,Y} = \ent{X} + \ent{Y}
\end{equation}
if and only if $X,Y$ are independent, and
\begin{equation}\label{cond-dec} \ent{X | Y} \leq \ent{X}\end{equation} always.

Suppose now that $(X,Y,Z)$ is a triple of random variables. Applying~\eqref{cond-dec} to $(X | Z = z)$ and summing over $z$ (weighted by $p_Z(z)$) gives
\begin{equation}\label{submodularity-first} \ent{X | Y, Z} \leq \ent{X | Z},\end{equation} which is known as \emph{submodularity}. It may equivalently be written as
\begin{equation}\label{submod-basic}  \ent{X,Y,Z} + \ent{Z} \leq \ent{X,Z} + \ent{Y,Z}.\end{equation}

The \emph{conditional mutual information} $\I{X:Y|Z}$ is defined by
\[
  \I{X:Y|Z} \coloneqq \sum_z p_Z(z) \I{ (X|Z=z) : (Y|Z=z)}.
\]
Submodularity is equivalent to the statement that
\begin{equation}
  \label{nonneg-cond}
  \I{X:Y|Z} \geq 0,
\end{equation}
since
\begin{equation}\label{cond-form-mutual-2}  \I{X:Y|Z} = \ent{X,Z} + \ent{Y,Z} - \ent{X,Y,Z} - \ent{Z}.\end{equation}
Equality occurs in~\eqref{nonneg-cond} (and hence in~\eqref{submod-basic}) if and only if $X,Y$ are conditionally independent relative to $Z$.

\subsubsection*{$G$-valued random variables}
We turn now to additional properties enjoyed by random variables taking values in an additive group $G$.

Whenever $X,Y$ are $G$-valued random variables, we have
that
\[
  \ent{X\pm Y} \geq \ent{X\pm Y|Y} = \ent{X|Y} = \ent{X} - \I{X:Y}
\]
and similarly with the roles of $X,Y$ reversed, thus
\begin{equation}
  \label{sumset-lower-gen}
  \max(\ent{X}, \ent{Y}) - \I{X:Y} \leq \ent{X\pm Y}.
\end{equation}
We note also the conditional variant of~\eqref{sumset-lower-gen}, namely
\[
  \max(\ent{X|Z}, \ent{Y|Z}) - \I{X:Y|Z} \leq \ent{X\pm Y|Z},
\]
which follows from~\eqref{sumset-lower-gen} by conditioning on $Z = z$ and summing over $z$ (weighted by $p_Z(z)$).

A particular consequence of~\eqref{sumset-lower-gen} is that
\begin{equation}
  \label{sumset-lower}
  \max(\ent{X}, \ent{Y}) \leq \ent{X\pm Y}
\end{equation}
when $X,Y$ are independent.

Continuing to suppose that $X, Y$ are independent, recall the definition~\eqref{ruz-dist-def} of the Ruzsa distance $\dist{X}{Y}$ which, since $X$ and $Y$ are independent, is that
\[
  \dist{X}{Y} = \ent{X - Y} - \tfrac{1}{2} \ent{X} - \tfrac{1}{2} \ent{Y}.
\]
Comparing this with~\eqref{sumset-lower} we see that
\begin{equation}
  \label{ruzsa-diff}
  |\ent{X}-\ent{Y}| \leq 2\dist{X}{Y}.
\end{equation}
We may also deduce that
\[
  \ent{X-Y} - \ent{X}, \ent{X-Y} - \ent{Y} \leq 2\dist{X}{Y}.
\]

The most important property of the Ruzsa distance is the (Ruzsa) triangle inequality
\[ \dist{X}{Y} \leq \dist{X}{Z} + \dist{Z}{Y}.\]
This was shown in~\cite{ruzsa-entropy} and~\cite[(16)]{tao-entropy}; we recall a proof for completeness.
This is equivalent to establishing
\begin{equation}\label{submod-explicit} \ent{X - Y} \leq \ent{X-Z} + \ent{Z-Y} - \ent{Z}\end{equation}
whenever $X, Y, Z$ are independent. To prove this, apply~\eqref{nonneg-cond} with a change of variables to get $\I{X-Z : Y | X - Y} \geq 0$ which, when written out in full, gives
\[ \ent{X - Z, X - Y} + \ent{Y, X - Y} \geq \ent{X - Z, Y, X - Y} + \ent{X - Y}.\]
Using
\[ \ent{X - Z, X - Y} = \ent{X-Z, Y-Z}  \leq \ent{X - Z} + \ent{Y - Z},\]
\[ \ent{Y, X - Y} = \ent{X, Y}, \] and
\[ \ent{X - Z, Y, X - Y} = \ent{X, Y, Z} = \ent{X, Y} + \ent{Z},\] and rearranging, we indeed obtain~\eqref{submod-explicit}. (As observed in~\cite{gmt}, we do not in fact use the independence of $X$ and $Y$ here.)

We will also need conditional variants of the distance.
If $(X, Z)$ and $(Y, W)$ are random variables (where $X$ and $Y$ are $G$-valued) we define
\begin{equation} \dist{X  | Z}{Y | W}   \coloneqq \sum_{z,w} p_{Z}(z) p_{W}(w) \dist{(X| Z=z)}{(Y| W = w )}.\label{cond-dist-def}
\end{equation}
Alternatively, if $(X',Z'), (Y',W')$ are independent copies of the variables $(X,Z)$, $(Y,W)$,
\begin{equation}\label{cond-dist-alt} \dist{X  | Z}{Y | W} = \ent{X'-Y'|Z',W'} - \tfrac{1}{2} \ent{X'|Z'} - \tfrac{1}{2}\ent{Y'|W'} .\end{equation}

To conclude this appendix we give the proofs of two results from the literature which were used in the main text.  The first is an inequality of Madiman~\cite[Theorem I]{madiman}, which is a more general form of an inequality of Kaimanovich and Vershik~\cite[Proposition 1.3]{kv}.

\begin{lemma}
Suppose that $X, Y, Z$ are independent random variables taking values in an abelian group. Then
\[
  \ent{X + Y + Z} - \ent{X + Y} \leq \ent{Y+Z} - \ent{Y}.
\]
\end{lemma}
\begin{proof}
By~\eqref{cond-form-mutual-2} we have
\begin{align*}
\I{ X : Z | X+Y+Z} &= \ent{X, X+Y+Z} + \ent{Z, X+Y+Z} \\
&\quad - \ent{X, Z, X+Y+Z} - \ent{X+Y+Z}.
\end{align*}
However, using~\eqref{indep} three times we have $\ent{X, X+Y+Z} = \ent{X, Y+Z} = \ent{X} + \ent{Y+Z}$, $\ent{Z, X+Y + Z} = \ent{Z, X+Y} = \ent{Z} + \ent{X+Y}$ and $\ent{X, Z, X+Y+Z} = \ent{X, Y, Z} = \ent{X} + \ent{Y} + \ent{Z}$.

After a short calculation, we see that the claimed inequality is equivalent to the assertion that $\I{ X : Z | X+Y+Z} \geq 0$, which of course is an instance of~\eqref{nonneg-cond}.
\end{proof}

The next lemma is not quite in the literature but is very closely related to the entropic version of the Balog--Szemer\'edi--Gowers lemma due to the fourth author~\cite[Lemma 3.3]{tao-entropy}. Here we provide slightly better constants and a slightly simpler proof.
\begin{lemma}\label{lem-bsg}
  Let $(A,B)$ be a $G^2$-valued random variable, and set $Z \coloneqq A+B$.
Then
\begin{equation}\label{2-bsg-takeaway} \sum_{z} p_Z(z) \dist{(A | Z = z)}{(B | Z = z)} \leq 3\I{A:B} + 2 \ent{Z} - \ent{A} - \ent{B}. \end{equation}
\end{lemma}
We stress that the quantity $2 \ent{Z} - \ent{A} - \ent{B}$ is \emph{not} the same as $2\dist{A}{B}$, because $(A,B)$ are given a joint distribution which may not be independent. In particular, $\ent{Z}=\ent{A+B}$ may not match the entropy of a sum of independent copies of $A$ and $B$.
\begin{proof}
In the proof we will need the notion of \emph{conditionally independent trials} of a pair of random variables $(X,Y)$ (not necessarily independent). We say that $X_1, X_2$ are conditionally independent trials of $X$ relative to $Y$ by declaring $(X_1 | Y = y)$ and $(X_2 | Y = y)$ to be independent copies of $(X | Y = y)$ for all $y$ in the range of $Y$.
We then have
\[ \ent{(X_1 | Y = y), (X_2 | Y = y)} = 2\ent{X | Y = y}\] for all $y$, which upon summing over $y$ (weighted by $p_Y(y)$) gives \[ \ent{X_1, X_2 | Y} = 2 \ent{X | Y}\] and hence
\begin{align}\nonumber
\ent{X_1, X_2, Y}  = \ent{X_1, X_2 | Y } + \ent{Y} & = 2 \ent{X | Y} + \ent{Y} \\ & = 2 \ent{X,Y} - \ent{Y}.\label{cond-trial-h}
\end{align}
Note also that the marginal distributions of $(X_1,Y)$ and $(X_2,Y)$ each match the original distribution $(X,Y)$.

Turning to the proof of \Cref{lem-bsg} itself, let $(A_1, B_1)$ and $(A_2, B_2)$ be conditionally independent trials of $(A,B)$ relative to $Z$, thus $(A_1,B_1)$ and $(A_2,B_2)$ are coupled through the random variable $A_1 + B_1 = A_2 + B_2$, which by abuse of notation we shall also call $Z$.

Observe that the left-hand side of~\eqref{2-bsg-takeaway} is
\begin{equation}\label{lhs-to-bound}
\ent{A_1 - B_2| Z} - \tfrac{1}{2}\ent{A_1 | Z} - \tfrac{1}{2} \ent{B_2 | Z}.
\end{equation}
since, crucially, $(A_1 | Z=z)$ and $(B_2 | Z=z)$ are independent for all $z$.

Applying submodularity~\eqref{submod-basic} gives
\begin{equation}\label{bsg-31} \begin{split}
&\ent{A_1 - B_2} + \ent{A_1 - B_2, A_1, B_1} \\
&\qquad \leq \ent{A_1 - B_2, A_1} + \ent{A_1 - B_2,B_1}.
\end{split}\end{equation}
We estimate the second, third and fourth terms appearing here.
First note that, by~\eqref{cond-trial-h} (noting that the tuple $(A_1 - B_2, A_1, B_1)$  determines the tuple $(A_1, A_2, B_1, B_2)$ since $A_1+B_1=A_2+B_2$)
\begin{equation}\label{bsg-24} \ent{A_1 - B_2, A_1, B_1} = \ent{A_1, B_1, A_2, B_2} = 2\ent{A,B} - \ent{Z}.\end{equation}
Next observe that
\begin{equation}\label{bsg-23} \ent{A_1 - B_2, A_1} = \ent{A_1, B_2} \leq \ent{A} + \ent{B}.
\end{equation}
Finally, we have
\begin{equation}\label{bsg-25} \ent{A_1 - B_2, B_1} = \ent{A_2 - B_1, B_1} = \ent{A_2, B_1} \leq \ent{A} - \ent{B}.\end{equation}
Substituting~\eqref{bsg-24},~\eqref{bsg-23} and~\eqref{bsg-25} into~\eqref{bsg-31} yields
\[ \ent{A_1 - B_2} \leq 2\I{A:B} + \ent{Z}\] and so by~\eqref{cond-dec}
\[\ent{A_1 - B_2 | Z}  \leq 2\I{A:B} + \ent{Z}.\]
Since
\begin{align*} \ent{A_1 | Z} & = \ent{A_1, A_1 + B_1} - \ent{Z} \\ & = \ent{A,B} - \ent{Z} \\ & = \ent{A} + \ent{B} - \I{A:B} - \ent{Z}\end{align*}
and similarly for $\ent{B_2 | Z}$, we see that~\eqref{lhs-to-bound} is bounded by
$3\I{A:B} + 2\ent{Z}-\ent{A}-\ent{B}$ as claimed.
\end{proof}

\section{From entropic PFR to combinatorial PFR}\label{equiv-app}
In this appendix we repeat the arguments from~\cite{gmt} showing that \Cref{pfr-entropy} with some constant $C'$ in place of $11$ implies \Cref{pfr} with $C=C'+1$.
In particular, with $C'=11$ in \Cref{pfr-entropy}, this gives the claimed constant $C=12$ for \Cref{pfr}.

Let $A,K$ be as in \Cref{pfr}.  Let $U_A$ be the uniform distribution on $A$, thus $\ent{U_A} = \log |A|$. By~\eqref{convexity-bound} and the fact that $U_A + U_A$ is supported on $A + A$, $\ent{U_A + U_A} \leq \log|A+A|$. The doubling condition $|A+A| \leq K|A|$ therefore gives
\[ \dist{U_A}{U_A} \leq \log K.\]
By \Cref{pfr-entropy}, we may thus find a subspace $H$ of $\F_2^n$ such that
\begin{equation}\label{uauh} \dist{U_A}{U_H} \leq \tfrac{1}{2} C' \log K.\end{equation}
By~\eqref{ruzsa-diff} we conclude that
\begin{equation}\label{ah}
  |\log |H| - \log |A|| \leq C' \log K.
\end{equation}
From definition of Ruzsa distance,~\eqref{uauh} is equivalent to
\[ \ent{U_A - U_H} \leq \log( |A|^{1/2} |H|^{1/2}) + \tfrac{1}{2} C' \log K.\]
By~\eqref{px-lower} we conclude the existence of a point $x_0 \in \F_p^n$ such that
\[ p_{U_A-U_H}(x_0) \geq |A|^{-1/2} |H|^{-1/2} K^{-C'/2},\]
or equivalently
\[ |A \cap (H + x_0)| \geq K^{-C'/2} |A|^{1/2} |H|^{1/2}.\]
Applying the Ruzsa covering lemma~\cite[Lemma 2.14]{tao-vu}, we may thus cover $A$ by at most
\[ \frac{|A + (A \cap (H+x_0))|}{|A \cap (H + x_0)|} \leq \frac{K|A|}{K^{-C'/2} |A|^{1/2} |H|^{1/2}} = K^{C'/2+1} \frac{|A|^{1/2}}{|H|^{1/2}}\]
translates of
\[ \bigl(A \cap (H + x_0)\bigr) - \bigl(A \cap (H + x_0)\bigr) \subseteq H.\]
If $|H| \leq |A|$ then we are already done thanks to~\eqref{ah}.  If $|H| > |A|$ then we can cover $H$ by at most $2 |H|/|A|$ translates of a subspace $H'$ of $H$ with $|H'| \leq |A|$.  We can thus cover $A$ by at most
\[ 2K^{C'/2+1} \frac{|H|^{1/2}}{|A|^{1/2}}\]
translates of $H'$, and the claim again follows from~\eqref{ah}.


\begin{thebibliography}{99}
\bibitem{aaronson}
J. Aaronson, \emph{A counterexample to a strong variant of the polynomial Freiman--Ruzsa conjecture}, preprint available at \url{https://arxiv.org/abs/1902.00353}

\bibitem{bourgain-chang} J.~Bourgain and M.-~C.~Chang, \emph{On the size of $k$-fold sum and product sets of integers,} J. Amer. Math. Soc. \textbf{17} (2004), no. 2, 473--497.


 \bibitem{chang}
M.~C.~Chang, \emph{Some consequences of the polynomial Freiman--Ruzsa conjecture}, C. R. Math. Acad. Sci. Paris \textbf{347} (2009), no.11--12, 583--588.

\bibitem{diao}
H. Diao, \emph{Freiman--Ruzsa-type theory for small doubling constant}, Math. Proc. Cambridge Philos. Soc. \textbf{146} (2009), no.2, 269--276.


\bibitem{even-zohar}
C.~Even-Zohar, \emph{On sums of generating sets in $\Z_2^n$}, Combin. Probab. Comput. \textbf{21} (2012), no. 6, 916--941.


  \bibitem{ggmt-pfr-odd} W.~T.~Gowers, B.~J.~Green, F.~R.~W.~M.~Manners and T.~C.~Tao, \emph{The polynomial Freiman--Ruzsa conjecture in odd characteristic,} manuscript.



  \bibitem{green-fin-fields} B.~J.~Green, \emph{Finite field models in additive combinatorics,} in Surveys in Combinatorics 2005  London Math. Soc. Lecture Notes \textbf{327}, 1--27.

  \bibitem{green-pfr-note} \bysame, \emph{Notes on the polynomial Freiman--Ruzsa conjecture,} unpublished note available at \url{https://people.maths.ox.ac.uk/greenbj/papers/PFR.pdf}

\bibitem{gt-freiman}  B.~J.~Green, T.~C.~Tao, \emph{Freiman's theorem in finite fields via extremal set theory}, Combin. Probab. Comput. \textbf{18} (2009), 335--355.

\bibitem{gt-equiv}
\bysame, \emph{An equivalence between inverse sumset theorems and inverse conjectures for the $U^3$ norm}, Math. Proc. Cambridge Philos. Soc. \textbf{149} (2010), no.1, 1--19.

  \bibitem{gmt} B.~J.~Green, F.~R.~W.~M.~Manners and T.~C.~Tao, \emph{Sumsets and entropy revisited,} preprint available at \url{https://arxiv.org/abs/2306.13403}



  \bibitem{hosseini-lovett}
  K. Hosseini, S. Lovett, \emph{A bilinear Bogolyubov--Ruzsa lemma with polylogarithmic bounds}, Discrete Anal. \textbf{2019} (2019), Paper No. 10, 14 pp.

  \bibitem{kv} V.~A.~Kaimanovich and A.~M.~Vershik, \emph{Random walks on discrete groups: boundary and entropy}, Ann. Probab. \textbf{11} (1983), no. 3, 457--490.
  \bibitem{kk} N.~H.~Katz and P.~Koester, \emph{On additive doubling and energy,} SIAM J. Discrete Math. \textbf{24} (2010), no. 4, 1684--1693.

 \bibitem{keevash-lifshitz}
 P.~Keevash and N.~Lifshitz, \emph{Sharp hypercontractivity for symmetric groups and its applications}, preprint available at \url{https://arxiv.org/abs/2307.15030}.


  \bibitem{konyagin}
S.~Konyagin, \emph{On the Freiman theorem in finite fields}, Math. Notes \textbf{84} (2008), 435--438.

  \bibitem{liao1}
  J.~Liao, comment available at \url{https://terrytao.wordpress.com/2023/11/13/on-a-conjecture-of-marton/comment-page-1/\#comment-682353}

  \bibitem{liao2}
  J.~Liao, in preparation.

  \bibitem{lovett}
S.~Lovett, \emph{Equivalence of polynomial conjectures in additive combinatorics},
Combinatorica \textbf{32} (2012), no. 5, 607--618.

\bibitem{lovett-survey}
\bysame, \emph{An Exposition of Sanders' Quasi-Polynomial Freiman--Ruzsa Theorem}, Theory of Computing Library Graduate Surveys \textbf{6} (2015), 1--14.

\bibitem{lovett-regev}
S.~Lovett and O.~Regev, \emph{A counterexample to a strong variant of the polynomial Freiman-Ruzsa conjecture in Euclidean space}, Discrete Anal.\textbf{2017} (2017), Paper No. 8, 6 pp.

\bibitem{madiman}
M. Madiman, \emph{On the entropy of sums}, in Information Theory Workshop 2008. ITW '08.  IEEE:\ 303--307, 2008.

\bibitem{manners}
F.~Manners, \emph{Finding a low-dimensional piece of a set of integers}, Int. Math. Res. Not. IMRN \textbf{2017} (2017), 4673--4703.

\bibitem{manners-pfr-formulation} F.~Manners, \emph{Formulations of the PFR Conjecture over $\Z$}. Math. Proc. Cambridge Philos. Soc. \textbf{166} (2019), no. 2, 243--245.

\bibitem{mudgal}
A. Mudgal, \emph{An Elekes--R\'onyai theorem for sets with few products}, preprint available at \url{https://arxiv.org/abs/2308.04191}.

\bibitem{pz}
D.  P\'alv\H{o}lgyi and D. Zhelezov, \emph{Query complexity and the polynomial Freiman-Ruzsa conjecture},
Adv. Math. \textbf{392} (2021), Paper No. 108043, 18 pp.

  \bibitem{plunnecke}
H.~Pl\"unnecke, \emph{Eine zahlentheoretische Anwendung der Graphtheorie}, J. reine angew. Math. \textbf{243} (1970), 171--183.

  \bibitem{ruzsa-fin-field} I.~Z.~Ruzsa, \emph{An analog of Freiman’s theorem in groups}, Ast\'erisque \textbf{258} (1999), 323--326.

  \bibitem{ruzsa-sumsets-survey} I.~Z.~Ruzsa, \emph{Sumsets and structure}, Adv. Courses Math. CRM Barcelona, Birkh\"auser Verlag, Basel, 2009, 87--210.

  \bibitem{ruzsa-entropy} \bysame, \emph{Sumsets and entropy,} Random Struct. Alg., \textbf{34} (2009), 1--10.


  \bibitem{samorodnitsky} A.~Samorodnitsky, \emph{Low-degree tests at large distances,} In Proceedings of the thirty-ninth annual ACM symposium on Theory of computing, 2007, 506--515.
  \bibitem{sanders} T.~Sanders, \emph{On the Bogolyubov–Ruzsa Lemma,} Analysis and PDE \textbf{5} (2012), no. 3, 627--655.
  \bibitem{sanders2} \bysame, \emph{The structure theory of set addition revisited,} Bull. Amer. Math. Soc. \textbf{50} (2013), 93--127.
\bibitem{tao-regularity}
T.~C.~Tao, \emph{Szemer\'edi's regularity lemma revisited}, Contrib. Discrete Math. \textbf{1} (2006), no. 1, 8--28.

\bibitem{tao-poincare}
\bysame, Poincar\'e's legacies, pages from year two of a mathematical blog. Part I. American Mathematical Society, 2009.


  \bibitem{tao-entropy} \bysame, \emph{Sumset and inverse sumset theory for Shannon entropy}, Combin. Probab. Comput. \textbf{19} (2010), no. 4, 603--639.
  \bibitem{tao-vu}
T.~C. Tao and V. Vu, Additive combinatorics, Cambridge Stud. Adv. Math., 105
Cambridge University Press, Cambridge, 2006, xviii+512 pp.

\end{thebibliography}
\end{document}